\documentclass[11 pt, reqno]{amsart}
\usepackage[foot]{amsaddr}

\usepackage{amsmath}
\usepackage{amsthm}
\usepackage[english]{babel}
\usepackage{amssymb}
\usepackage{graphics}
\usepackage{epsfig}
\usepackage{amscd}
\usepackage{amsfonts}
\usepackage{amsbsy}
\usepackage[normalem]{ulem}
\usepackage{tikz}
\usepackage{soul}
\usepackage[font=small,labelfont={bf},labelsep=period,width=1.2\textwidth]{caption}
\usepackage{float}
\usepackage{subcaption}
\usepackage{overpic}
\usepackage{dsfont}
\usepackage{tabularx}
\usepackage{array}
\usepackage{makecell}
\usepackage{threeparttable}
\newcommand{\R}{\mathds{R}}
\newcommand{\C}{\ensuremath{\mathbb{C}}}

\newtheorem {theorem} {Theorem}
\newtheorem {prop} [theorem] {Proposition}
\newtheorem {corollary}{Corollary}
\newtheorem {lemma} [theorem] {Lemma}

\newtheorem {remark} {Remark}
\newtheorem* {remark*} {Remark}

\usepackage{color}

\usepackage{fancyhdr}
\begin{document}

\author[Pablo Amster$^{1,2}$, Andr\'es Rivera$^3$, Sebasti\'an Pedersen$^1$]
{Pablo Amster$^{1,2}$, Andr\'es Rivera$^3$, Sebasti\'an Pedersen$^1$}

\title[Delay-induced dynamics in a nonlinear crime interaction model with periodic forcing]{Delay-induced dynamics in a nonlinear crime interaction model with periodic forcing}



\address{$^1$ Departamento de Matem\'atica, Facultad de Ciencias Exactas y Naturales, 
Universidad de Buenos Aires. Ciudad Universitaria Pab. I, 
(1428) Buenos Aires, Argentina.}
\address{$^2$ Imas-Conicet, Facultad de Ciencias Exactas y Naturales, 
Universidad de Buenos Aires. Ciudad Universitaria Pab. I, 
(1428) Buenos Aires, Argentina.}
\address{$^3$ Departamento de Ciencias Naturales y Matem\'aticas
Pontificia Universidad Javeriana Cali, Facultad de Ingenier\'ia y Ciencias,
Calle 18 No. 118--250 Cali, Colombia.}

\email{pamster@dm.uba.ar, amrivera@javerianacali.edu.co, spedersen@dm.uba.ar}



\date{}
\dedicatory{}
\begin{abstract}
A nonlinear time-delay model is proposed to describe the interaction dynamics between criminal and non-criminal populations, combining social influence mechanisms, saturation effects represented by a Holling type II functional response, and time-dependent law-enforcement actions. The delay accounts for the latency between exposure to criminal behavior and behavioral response, introducing memory effects that naturally lead to a framework of delay differential equations. Fundamental analytical properties, including positivity, global existence, and invariance of the feasible region, are established to ensure the mathematical consistency of the population interpretation. In the autonomous setting, explicit threshold conditions governing the stability of the criminal-free equilibrium and the emergence of coexistence states are derived, while the delay is shown to induce stability switches and oscillatory regimes through characteristic-root crossings. In the non-autonomous case, topological degree arguments guarantee the existence of strictly positive periodic solutions, with the resulting conditions involving the average law-enforcement intensity measures rather than short-term fluctuations. These results highlight the role of time delay as a structural mechanism underlying recurrent patterns and complex temporal behavior in crime dynamics. 
\end{abstract}

\maketitle


\section{Introduction} \label{sec:intro}
In recent years, various approaches have been developed to analyze criminal behavior within a population, with the aim of developing models and/or strategies that contribute to understanding and mitigation of criminal behavior and promote safer social development. These approaches include, among others, population dynamics, population behavior theory, optimal control theory, and stochastic modeling (see \cite{SOOKNANAN2023128212} and the references therein, and see \cite{Crokidakis2024DynamicsOD, Gutema2024AMA, Kotola2022AMM, Teklu2022MathematicalMI, Sarmah2025ExploringTD, Short2008ASM, McMillon2025, Oluwasegun2022, kwofie2023curtailing, mamo2024, mohammad2017analysis, srivastav2019modeling, calatayud2025dynamical}).

In terms of mathematical modeling, one approach consists of models inspired by the classical predator-prey framework (see, for example, \cite{ ABBAS2017121,TRIPATHI2021125725} and the references therein) in which the total population is divided into two groups. Group $C$ consists of individuals who commit crimes, whereas Group $N$ consists of individuals who do not commit crimes. A key motivation for adopting compartmental models is the assumption that criminal behavior can be regarded as socially contagious through association with delinquent peers  (see, for example, references [2, 3] in \cite{SOOKNANAN2023128212}). We then adopt the usual assumptions underlying compartmental population models. 

Continuing this line of research, predator-prey models have also been used to describe interactions between criminal and non-criminal populations. In this context, nonlinear functional responses provide a natural way of incorporating saturation effects into the interaction mechanism. More broadly, the review by Sooknanan and Seemungal \cite{SOOKNANAN2023128212} shows that most epidemiological models of criminal behavior rely on mass-action or standard-incidence contact rates, with mass action being the most frequently used formulation. Time-delay effects have also been incorporated into crime-dynamics models, although with different interpretations of the delayed process. To situate the present formulation within this literature, Table \ref{tab:model comparison} summarizes the main structural features of several models closely related to the present work.
\begin{table}[ht]
\centering
\begin{threeparttable}
\caption{Main structural features of selected crime-dynamics models 
closely related to the present work.}
\label{tab:model comparison}
\renewcommand{\arraystretch}{1.25}
\begin{tabularx}{\textwidth}{
    >{\raggedright\arraybackslash}p{2.1cm}
    >{\centering\arraybackslash}X
    >{\centering\arraybackslash}X
    >{\centering\arraybackslash}X
    >{\centering\arraybackslash}X
}
\hline
\textbf{Reference}
& \makecell{\textbf{Functional}\\\textbf{response}}
& \makecell{\textbf{Population}\\\textbf{growth}}
& \textbf{Delay}
& \makecell{\textbf{Law}\\\textbf{enforcement}} \\
\hline
\cite{ABBAS2017121}
& Linear
& Logistic
& \textbf{--}
& Constant \\
\cite{srivastav2019modeling}
& Mass action
& Recruitment/death
& Catching delay\tnote{a}
& \textbf{--} \\
\cite{SOOKNANAN2023128212}\tnote{b}\cite{TRIPATHI2021125725}
& Holling type II
& Logistic
& \textbf{--}
& Constant \\
\cite{Mebratie2021CrimeDynamics}
& Mass action\tnote{c}
& Recruitment/death
& \textbf{--}
& Dynamic police force\\
\textbf{Present work}
& Holling type II
& General $f(N)$
& Behavioral delay\tnote{d}
& $T$-periodic \\
\hline
\end{tabularx}
\begin{tablenotes}
\footnotesize
\item[a] The delay represents the time associated with catching criminals.
\item[b] See reference [75].
\item[c]  Criminal transmission includes bilinear contact terms, with an additional modification accounting for media coverage. 
\item[d] The delay represents the time between exposure to criminal behavior and the subsequent transition toward criminal behavior.
\end{tablenotes}
\end{threeparttable}
\end{table}

The preceding table highlights the main structural differences between the present formulation and closely related crime-dynamics models. The proposed model integrates a general growth function $f(N)$ for the non-criminal population, a $T$-periodic law-enforcement function, a Holling type II functional response, and a discrete delay, thereby extending previous autonomous formulations to a non-autonomous delay framework. Within this setting, we investigate the dependence of local stability on the delay in the autonomous case and establish sufficient conditions for the existence of strictly positive periodic solutions when the law enforcement is time-periodic.

More precisely, we propose the following time-delay mathematical model to describe the interaction dynamics between non-criminal and criminal individuals. Here, we denote by $N=N(t)$ the non-criminal population and by $C=C(t)$ the criminal population. The proposed system of delay differential equations is given by
\begin{equation}\label{main eq 1}
\begin{split}
\dot{N}(t)&=N(t)f(N(t))-\frac{\varphi N(t)C(t)}{\nu+N(t)}+\sigma N(t)C(t),\\
\dot{C}(t)&=-\eta C(t)-l_{e}(t)C(t) +\frac{\gamma N(t-\tau)C(t-\tau)}{\nu+N(t-\tau)}.
\end{split}
\end{equation}
Let us describe the components in \eqref{main eq 1}
\begin{itemize}
\item[1.] In the absence of $C$, the non-criminal population $N$ is assumed to follow a growth function $f$, where $f$ represents the relative growth rate of the non-criminal population and is assumed to be of generalized logistic type, satisfying $f(0)>0$, $f'(s)<0$ for $s>0$, and having a unique positive zero. For example, if $f(s)=\mu(m-s) $, then $N$ follows logistic growth in the absence of $C$. The function $\displaystyle{\varphi N/(\nu+N)}$ represents the rate at which individuals in $C$ victimize individuals in $N$, and the parameter $\varphi$ represents the maximum rate at which interactions with individuals in $C$ affect individuals in $N$. The term $\sigma N C$ represents a community-based protective or resilience response of the non-criminal population arising from interactions between the non-criminal and criminal populations. Specifically, we assume that the intensity of such responses per non-criminal individual increases proportionally with the presence of criminal individuals, at rate $\sigma C$, naturally resulting in a bilinear contribution at the population level. This term should be viewed as representing informal or community-level responses such as collective vigilance, social support, preventive behavior, and other mechanisms that reinforce the persistence of non-criminal behavior. Hence, the two interaction terms in the first equation describe competing effects: criminal victimization or influence, represented by $\varphi NC/(\nu+N)$, and the compensating protective response of the non-criminal population, represented by $\sigma N C$.
\\
\item[2.]  In the absence of $N$, the parameter $\eta$ is the natural mortality rate of $C$. To prevent the rapid growth of the criminal population, law-enforcement measures or interventions are incorporated through a positive and periodic function $l_e(t)$. The last term in the second equation of \eqref{main eq 1} represents the delayed transition of exposed non-criminal individuals into the criminal population. Here, $\tau\geq0$ represents the latency period between exposure to criminal behavior and the subsequent transition into the criminal population. Therefore, $\displaystyle{\gamma N/(\nu+N)}$ represents the rate at which interactions with the non-criminal population contribute to the transition into the criminal population.
\end{itemize}

This paper is organized as follows. Section \ref{preliminar} presents preliminary results necessary for the following sections. Section \ref{sec:linear stab auton} analyzes the equilibria and local stability in the autonomous case. Section \ref{sec: non auton periodic sols} investigates the existence of periodic solutions in the non-autonomous case. Section \ref{sec: strategy control crime} discusses the implications of the analytic results in the context of crime dynamics and control. Finally, Section \ref{sec:numerical sims} presents some numerical simulations illustrating the main analytical results corresponding to Sections \ref{sec:linear stab auton} and \ref{sec: non auton periodic sols}.

\begin{remark*}[the delay as an exposed population] 
Many classic epidemiological models of $SEIR$ type include an exposed population $E$, consisting of individuals who have been in contact with an infected individual but have not yet entered the infectious class $I$ (see \cite{ bjornstad2020seirs,brauer, deutsch,diekmann}). It is well known that delays can be used to represent such a latent stage without introducing an additional compartment. (see \cite{chen2023delays, kuang, smith, murray2007mathematical, rihan2021delay, saade, rossini2025general, articleDiekmann2025}). We give a brief interpretation and analysis of how this occurs in our model \eqref{main eq 1}.
\begin{em}
    
Let $E(t)$ represent the non-criminal population that, at time $t$, has been in contact with the criminal population, that is, has been exposed to crime, but has not yet become part of the criminal population. Assume that the dynamics are initiated over the initial interval $[-\tau,0].$ Then we have that:
$$E(0) = \int\limits_{-\tau}^{0} \varphi h(s) ds = \int\limits_{-\tau}^{0} \varphi \frac{N(s)C(s)}{\nu+N(s)} ds = \int\limits_{-\tau}^{0} \varphi \frac{N_0(s)C_0(s)}{\nu+N_0(s)} ds,
$$
since, according to the first equation of system \eqref{main eq 1}, $\displaystyle{\varphi h(s) := \varphi\frac{N(s)C(s)}{\nu+N(s)}}$ models the rate at which criminals come into contact with non-criminals; and where $N_0(t)$ and $C_0(t)$, $-\tau\leq t\leq 0$, are the initial conditions for $N$ and $C$ respectively.

Now, in the time interval $[0,\tau]$, some of the individuals exposed during the initial interval will subsequently enter the criminal population at the rate $\gamma h(s-\tau)$, while new exposed individuals will be generated again at a rate of $\varphi h(s)$.  Therefore,
$$
E(\tau) = E(0) - \int\limits_{0}^{\tau} \gamma h(s-\tau) ds +  \int\limits_{0}^{\tau} \varphi h(s) ds = \int\limits_{-\tau}^{0} \left(\varphi -\gamma\right) h(s) ds +  \int\limits_{0}^{\tau} \varphi h(s) ds.
$$
Following the same reasoning, for $t\geq0$, we obtain:
$$
E(t) = \int\limits_{-\tau}^{t-\tau} \left(\varphi -\gamma\right) h(s) ds +  \int\limits_{t-\tau}^{t} \varphi h(s) ds,
$$
Differentiating this last expression (with respect to $t$), we obtain
$$
\dot{E}(t) = \varphi h(t) - \gamma h(t-\tau) = \frac{\varphi N(t)C(t)}{\nu+N(t)} - \frac{\gamma N(t-\tau)C(t-\tau)}{\nu+N(t-\tau)},
$$
which yields the corresponding equation for the exposed population. Then, the delay $\tau$ is interpreted as the time it takes for the non-criminal population, after being in contact with the criminal population, to definitively become a criminal population. Thus, system \eqref{main eq 1} can be interpreted as an $NEC$-type model, with the exposed stage represented implicitly through the delay rather by an additional state variable. 
\end{em}
\end{remark*}

\section{Preliminary results}\label{preliminar}
This section presents some preliminary results that will be used throughout the subsequent sections. Most of these results are standard, but we provide more detailed statements and proofs adapted to the needs of the present work.

\subsection{Positivity and global existence of solutions}
In the present context, it is important to establish that a nonnegative initial conditions generate nonnegative solutions $X:=(N,C)$. We prove this for system \eqref{main eq 1} subject to the initial conditions
\[
\begin{split}
N(s)&=\Phi_1(s), \quad C(s)=\Phi_2(s),\\
\Phi_1(s),\Phi_2(s)&\geq 0, \,s \in [-\tau,0], \quad \Phi_1(0),\Phi_2(0)>0, 
\end{split}
\]
where $\Phi(s):=(\Phi_{1}(s),\Phi_{2}(s))\in C([-\tau,0],\mathbb{R}^{2}_{+})$ the Banach space of continuous functions mapping $[-\tau,0]$ into $\mathbb{R}^{2}_{+}$, with the norm $\|\Phi\|=\sup_{-\tau \leq s\leq 0}\|\Phi(s)\|
$ and $\mathbb{R}^{2}_{+}=\left\{(x_{1},x_{2})\in\mathbb{R}^2\,|\,x_i\geq 0)\right\}$. Moreover, since the right-hand side of \eqref{main eq 1} is locally Lipschitz continuous, there exists a unique local solution $X=X(t,\Phi)$ for every continuous initial function $\Phi$.

First, note that we can write the equation for $N$ as 
\[
\dot{N}(t) = h(t)N(t), \quad \text{with} \quad h(t):=f(N(t))-\frac{\varphi C(t)}{\nu+N(t)}+\sigma C(t),
\]
and note that
\begin{equation}\label{positivity for N}
\begin{split}
\frac{d}{dt} \left(\text{e}^{-H(t)}N(t)\right) = \text{e}^{-H(t)}(-h(t))N(t) + \text{e}^{-H(t)}\dot{N}(t) = 0, \quad \text{where} \quad H(t) = \int_0^t h(s)ds.
\end{split}
\end{equation}
So $\text{e}^{-H(t)}N(t) = N_0$ for some constant $N_0$, and then $N(t) = N_0 \text{e}^{H(t)}$. Hence, $N(t)\geq 0$ whenever $N_0=\Phi_{1}(0)\geq 0$, and $N(t)>0$ whenever $N_0>0.$  
In the same fashion, we can write the equation for $C$ as
\[
\dot{C}(t) = a(t)C(t) + b(t), \quad \text{with} \quad a(t):=-(\eta+l_{e}(t)),\quad b(t):=\frac{\gamma N(t-\tau)C(t-\tau)}{\nu+N(t-\tau)}.
\]
Once again, notice that
\begin{equation}\label{positivity for C}
\begin{split}
\frac{d}{dt} \left(\text{e}^{-A(t)}C(t)\right) = \text{e}^{-A(t)}(-a(t))C(t) + \text{e}^{-A(t)}\dot{C}(t) = \text{e}^{-A(t)}b(t), \quad \text{where} \quad A(t)=\int_{0}^{t}a(s)\,ds.
\end{split}
\end{equation} So, equation \eqref{positivity for C} implies that
\begin{equation}\label{positivity for C bis}
\begin{split}
C(t) = C_0\text{e}^{A(t)} + \text{e}^{A(t)}\int_0^t \text{e}^{-A(s)}b(s)ds.
\end{split}
\end{equation}
It follows from \eqref{positivity for C bis} that $C(t)$ remains positive whenever $C_0>0$. Indeed, if $t_0>0$ were the first zero of $C(t)$, then
\[
0=C_0\text{e}^{A(t_{0})} + \text{e}^{A(t_{0})}\int_0^{t_{0}} \text{e}^{-A(s)}b(s)ds,
\]
which is a contradiction since the right-hand side is strictly positive. Although positivity can be established directly for the present system, we also state a more general result that will be useful for this purpose. Following Theorem 3.4, Chapter 3, in \cite{smith}, we consider the following delay differential equation with an initial condition:
\begin{equation}\label{non-linear system one delay initial condition}
\begin{split}
\dot{x}(t) &= f\left(t,x(t),x(t-\tau)\right), \quad t\geq s\\
x_s &= \phi \in C\left([s-\tau,s],\R^n\right).
\end{split}
\end{equation}
Given $x=(x_1,\dots,x_n)\in\R^n$, we write $x\geq 0$ if $x_i\geq0$ for all $i$. Let $\mathbb{R}^n_+$ denote the set of all $x\in\mathbb{R}^n$ such that $x\geq0$.
\begin{theorem}\label{positivity theorem}
Consider the system \eqref{non-linear system one delay initial condition} and assume that $f$ and $D_x f$ are continuous on $\R\times\R^n_+\times\R^n_+$ (or on some suitable subset), and that
$$ \text{For every } i,t \text{ and for all } x,y\in\R^n_+, \quad x_i=0 \implies f_i(t,x,y)\geq 0,$$
where $f=(f_1,\dots,f_n)$. If the initial condition $\phi$ in \eqref{non-linear system one delay initial condition} satisfies $\phi\geq 0$, then the corresponding solution $x(t)$ of \eqref{non-linear system one delay initial condition} satisfies $x(t)\geq0$ for all $t$ where it is defined.
\end{theorem}
\begin{proof}
    The proof follows the lines of Theorem 3.4, Chapter 3, in \cite{smith}. The corresponding ODE result is first applied on $s\leq t\leq s+\tau$ and the argument then iterated over successive intervals of length $\tau$.
\end{proof}
\begin{remark}
    Looking at our system \eqref{main eq 1}, and setting $x=(x_1,x_2)$ where $x_1=x_1(t)=N(t)$, $x_2=x_2(t)=C(t)$, and $y=(y_1,y_2)$ where $y_1=y_1(t)=N(t-\tau), y_2=y_2(t)=C(t-\tau)$, the system \eqref{main eq 1} has the form
    $$g(t,x,y) = \left(x_1f(x_1)-\varphi\frac{x_1x_2}{\nu+x_1} + \sigma x_1x_2,\, -\eta x_2 -l_e(t)x_2 + \gamma\frac{y_2y_1}{\nu+y_1}  \right).$$
    Since $g$ and $D_xg$ are continuous, the standard existence and uniqueness result implies that, for any admissible initial condition, system \eqref{main eq 1} has a unique maximal solution on some interval $-\tau\leq t<A$.  Writing $g=(g_1,g_2)$, if $x_1=0$ then $g_1(t,x,y)=0$, and if $x_2=0$ then $g_2(t,x,y)=\dfrac{\gamma\, y_2y_1}{\nu+y_1}$. Thus, the hypotheses of Theorem \ref{positivity theorem} are satisfied. Therefore, a nonnegative initial condition gives rise to a nonnegative solution of system \eqref{main eq 1}. So we are in the conditions of Theorem \ref{positivity theorem}, and we can conclude that  nonnegative initial conditions gives rise to a nonnegative solutions for our system \eqref{main eq 1}.
\end{remark}

\begin{lemma} Any solution $X(t,\Phi)$ of \eqref{main eq 1} is globally defined for every $\Phi\in C([-\tau,0],\mathbb{R}^{2}_{+}) $
\end{lemma}
\begin{proof}
Let $t\in [0,\tau]$. Since $N(t-\tau)$ and $C(t-\tau)$ are positive, \eqref{positivity for C bis} provides the following estimate
\[
C(t)\leq \mathcal{C}_{1}, \quad \text{with} \quad \mathcal{C}_{1}= \max_{t}\left\{\Phi_{2}(0)e^{A(t)}+\frac{\gamma}{\nu}\int_{-\tau}^{t-\tau}\Phi_2(m)\,dm\right\}.
\]
From here, we can estimate $\dot{N}(t)$ as follows
\[
\dot{N}(t)\leq N(t)(f(N(t))+\sigma \mathcal{C}_1),
\]
Consequently, $N(t)\leq \max\{N_{0},\tilde{N}\}:=\mathcal{N}_1$ where $N_{0}=\Phi_{1}(0)$ and $\tilde{N}$ is the unique solution of $f(N)=-\sigma \mathcal{C}_1$. Then, $N(t)$ and $\dot{N}(t)$ are bounded on $[0,\tau]$ implying that $N(t)$ is defined in the whole interval $[0,\tau]$. The same argument can be applied successively on each interval $n\tau\leq t\leq(n+1)\tau$, $n\geq0$, yielding analogous estimates $\mathcal C_{n+1}$ and $\mathcal N_{n+1}$. This guarantees the global existence of solutions of \eqref{main eq 1} for all $t\in[-\tau,\infty)$.
\end{proof}

\begin{remark} Global existence can also be established by a simpler argument. Indeed,
\[
\dot{C}(t)\leq \gamma C(t-\tau), 
\]
which implies that $C(t)$ is globally defined. Moreover,
\[
\dot{N}(t)\leq (\sigma C(t)+k)N(t), \quad \text{with} \quad k=\max{f(N)},
\]
which in turn implies the global existence of $N(t)$.
\end{remark}


\subsection{Stability switches}\label{subsec stability switches} Following Theorem 1.4, Chapter 3, in \cite{kuang}, we present a useful result for the analysis of stability switches as the delay increases.
\begin{theorem}
    Let $f(\lambda,\tau) = \lambda^n + g(\lambda,\tau)$, where $g(\lambda,\tau)$ is an analytic function. Assume
    $$\alpha = \limsup_{\begin{smallmatrix} \text{Re } \lambda >0 & \\ |\lambda|\to\infty \end{smallmatrix} } \left|\lambda^{-n}g(\lambda,\tau)\right|<1.$$
    Then, as $\tau$ varies, the sum of the multiplicities of roots of $f(\lambda,\tau)=0$ in $\{\text{Re } z >0\}$ can change only if a root appears on or crosses the imaginary axis.
\end{theorem}
\begin{proof}
    The proof follows the lines of Theorem 1.4, Chapter 3, in \cite{kuang}. 
\end{proof}
\begin{remark}
    This result provides a useful tool for analyzing delay-dependent stability changes in system \eqref{main eq 1}, since it relates such changes to the appearance of purely imaginary roots in the associated characteristic equation (see Appendix A). 
\end{remark}

\subsection{Results for the autonomous linear case} 
Here, we present a general result for the following linear delay system
\begin{equation}\label{linear system with one delay}
\dot{Z}(t)=AZ(t)+BZ(t-\tau),
\end{equation}
where $A,B \in \mathbb{M}_{2\times 2}$ are constant real matrices,
$Z=Z(t,\tau)\in \mathbb{R}^{2}$, and $\tau \geq 0$. This analysis is based on
Theorem 4.1 in Chapter 3 of \cite{kuang}. Throughout this subsection, we assume that $B$ is singular, i.e., $\det B=0$.
By Lemma in Appendix A, for $\lambda=\lambda_1+i\lambda_2 \in \mathbb{C}$, the associated characteristic equation is given by:
\begin{equation}\label{characteristic new general}
P(\lambda,\tau)=0 \quad  \Leftrightarrow \quad p(\lambda)+q(\lambda)e^{-\lambda \tau}=0,
\end{equation}
where
\[
p(\lambda)=\lambda^{2}-\nu_1 \lambda+\nu_2 \quad \text{and} \quad q(\lambda)=\varsigma_1-\varsigma_2 \lambda,
\]
with
\[
\begin{split}
\nu_1=tr A, \quad \nu_2=\det A, \quad \text{and} \quad 
\varsigma_1=\det(a^1|b^2)+\det (b^1|a^2), \quad \varsigma_2=tr B.
\end{split}
\]
Here, $\displaystyle{\det(a^1|b^2)}$ denotes the matrix with the first column from $A$ and the second column from $B$. Similarly, for $\displaystyle{\det(b^1|a^2)}$. Observe that the analytic functions $p(\lambda)$ and $q(\lambda)$ satisfy the following:
\begin{itemize}
    \item[a.] $p(0)+q(0)=\nu_2+\varsigma_1.$
    \item[b.] $\overline{p(-i\lambda_2)}=p(i\lambda_2)$ and $\overline{q(-i\lambda_2)}=q(i\lambda_2).$
    \item[c.] $q(\lambda)$ has no purely imaginary roots.
    \item[d.] For $\lambda_1\geq 0$, we have the following
\[
\begin{split}
\limsup_{|\lambda|\to \infty}&\left\{\left|\frac{q(\lambda)}{p(\lambda)}\right| \right\}=\limsup_{|\lambda|\to \infty}\left\{\left|\frac{\varsigma_1-\varsigma_2\lambda}{ \lambda^2-\nu_1\lambda+\nu_2}\right|\right\}=0.\\
\end{split}
\]
\item[e.] Let $\displaystyle{F:\mathbb{R}\to \mathbb{R}}$, $s\mapsto F(s)$ be defined by 
\begin{equation}\label{polynomial function}
\begin{split}
F(s)&=|p(is)|^2-|q(is)|^2,\\
&=|s^2+i\nu_1s-\nu_2|^2-|\varsigma_1-i\varsigma_2s|^2,\\
&=s^4+(\nu^2_1-2\nu_2-\varsigma^2_2)s^2 + \nu^2_{2}-\varsigma^2_1.
\end{split}
\end{equation}
Clearly, $\displaystyle{F(s)}$ has a finite number of zeros, and furthermore
\begin{equation*}
F(0)=\nu^2_{2}-\varsigma^2_{1}, \quad F^{\prime}(s)=2s(2s^2+(\nu^2_1-2\nu_2-\varsigma^2_2))\quad 
\text{and} \quad \lim_{s\to \pm \infty}F(s)=\infty.
\end{equation*}
The following lemma summarizes some relevant properties of $F(s)$.
\begin{lemma} Consider the real polynomial function $F(s)$ given by \eqref{polynomial function} and let $h=\nu^{2}_1-2\nu_2-\varsigma^2_{2}$. A graphical summary of these cases is provided in Figure~\ref{fig roots linear}.
\begin{itemize}
    \item[I.] If $h\geq 0$ and $F(0)\geq 0$, $F(s)$  has no positive roots.
      \item[II.] If $h<0$ and $F(0)\geq 0$, $F(s)$ has at most two positive roots, possibly none. 
      \begin{itemize}
          \item If $F(0)=0$, $F(s)$ has exactly one positive root at $\hat{s}_{\ast}=\sqrt{-h}$ with $F^{\prime}(\hat{s}_{\ast})>0.$
          \item If $0<F(0)<h^2/4$, $F(s)$ has exactly two positive roots, $0<\hat{s}_1 < \hat{s}_2$ given by 
          \[
\sqrt{2}\hat{s}_{k}=\sqrt{-h+(-1)^{k}\sqrt{{h}^2-4F(0)}}, \quad k=1,2,
          \]
          and such that
          \[
          \text{sgn}(F^{\prime}(\hat{s}_k))=(-1)^k, \quad k=1,2.
          \]
          \item If $F(0)=h^2/4$, $F(s)$ has exactly one positive root $\hat{s}=\sqrt{-h/2}$, with $F^{\prime}(\hat{s})=0$.
        \item If $F(0)>{h}^2/4$, $F(s)$ has no (positive) real roots.
          \end{itemize}
                \item[III.] If $F(0)<0$, $F(s)$ has exactly one positive root $\hat{s}_{2}$ with $\text{sgn}(F^{\prime}(\hat{s}_2))>0$ ($\hat{s}_2$ is that of the previous case with $k=2$.)
      \end{itemize}
\end{lemma}
\begin{figure}[ht]
\centering
\includegraphics[width=0.7\textwidth]{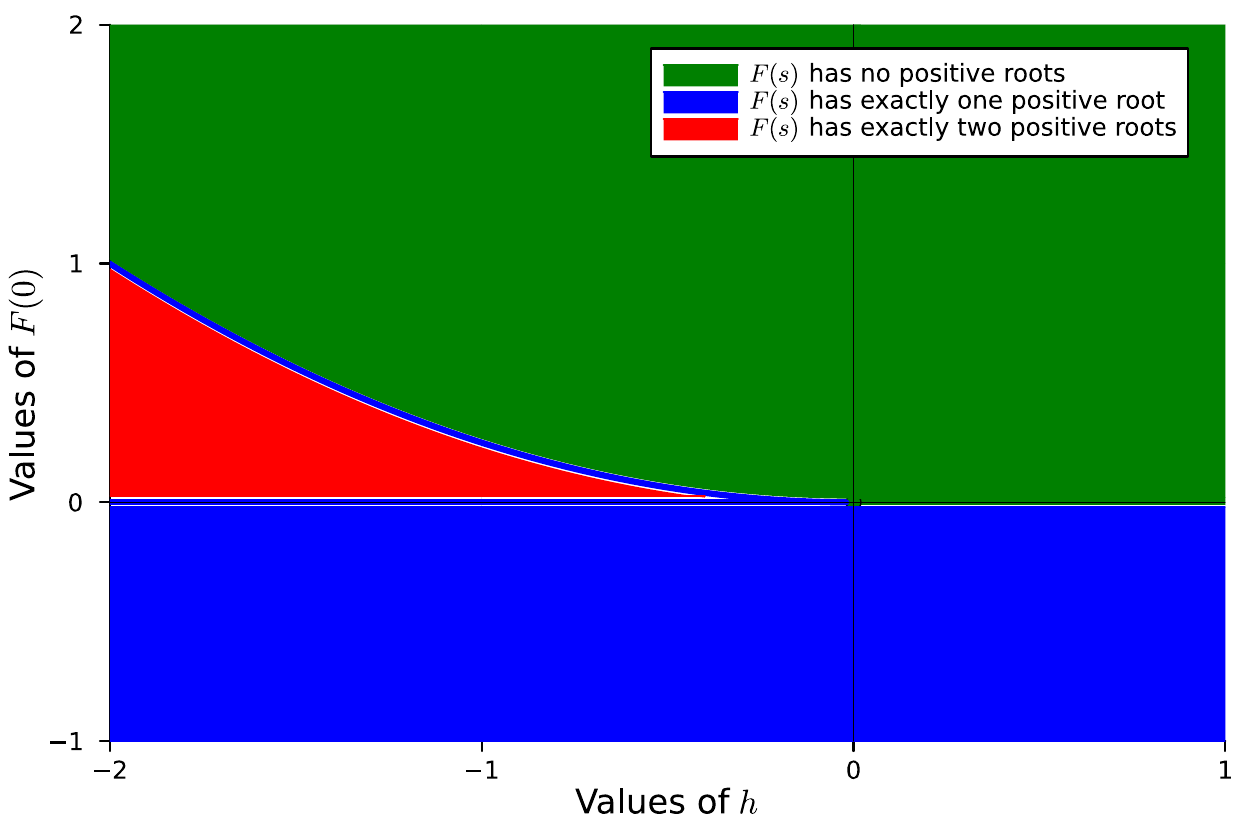}
\caption{Diagram of the regions of the positive roots of $F(s)$ according to the values of $h$ and $F(0)$ (see equation \eqref{polynomial function} and what follows). The blue curve corresponds to $F(0)=h^2/4$.}
\label{fig roots linear}
\end{figure}
\begin{proof}
Let $x=s^2$, and consider the quadratic function $G(x)=F(s(x))$ given by
\[
G(x)=x^2+h x+F(0), \quad \text{with} \quad h=\nu^{2}_1-2\nu_2-\varsigma^2_{2}.
\]
The roots of $F(s)$ are $s=\pm\sqrt{x}$ with
\[
2x=-h\pm\sqrt{h^2-4F(0)},
\]
In Case I, the real part of $x$ is non-positive; hence, $F(s)$ has no positive real roots. 
In Case II, the real part of $x$ is positive, then $F(s)$ has at most two positive real roots $s_k$ given by
\[
\sqrt{2}s_{k}=\sqrt{-h+(-1)^{k}\sqrt{h^2-4F(0)}}, \quad \text{if} \quad 0\leq F(0)\leq  h^2/4, \quad k=1,2.
\]
In particular,
\begin{itemize}
    \item if $F(0)=0$ then $\hat{s}_{\ast}=\sqrt{-h}$ is the only positive root of $F(s)$. 
    \item if $F(0)=h^2/4$ then $\hat{s}_{\ast \ast}=\sqrt{-h/2}$ is the only positive root of $F(s)$. 
\end{itemize}
Moreover, for any positive root $\hat{s}$ of $F(\hat{s})$ we have $F^{\prime}(\hat{s})=2\hat{s}G^{\prime}(\hat{x})$ with $\hat{x}=\hat{s}^2$. Then $\text{sgn}(F^{\prime}(\hat{s}))=\text{sgn}(G^{\prime}(\hat{x}))$
It follows that
\[
\begin{split}
\text{sgn}(F^{\prime}(\hat{s}_{\ast}))&=\text{sgn}(G^{\prime}(\hat{x}_{\ast}))=\text{sgn}(-h)>0, \quad F^{\prime}(\hat{s}_{\ast})=2\sqrt{-h^3}>0,\\
\text{sgn}(F^{\prime}(\hat{s}_{k}))&=\text{sgn}(G^{\prime}(\hat{x}_{k}))=\text{sgn}((-1)^{k}\sqrt{h^2-F(0)})=(-1)^{k}, \quad k=1,2,\\
\end{split}
\]
and $F^{\prime}(\hat{s}_{\ast \ast})=\hat{s}_{\ast \ast}G^{\prime}(\hat{x}_{\ast \ast})=0.$ Finally, if $F(0)<0$ then $\displaystyle{\hat{s}_2}$ is the only positive real root which as we know $F^{\prime}(\hat{s}_2)>0$.
\end{proof}
\end{itemize}
By combining the preceding calculations with Theorem 4.1 in Chapter 3 of \cite{kuang}, we obtain the following result.
\begin{theorem}\label{stability theorem}
Let $\lambda(\tau)=\lambda_1(\tau)+i \lambda_2(\tau) \in \C$ be a solution of \eqref{characteristic new general}, i.e., $P(\lambda(\tau),\tau)=0$. Let
\[
F(0)=\nu^2_{2}-\varsigma^{2}_1, \quad \text{and} \quad h=\nu^{2}_1-2\nu_2-\varsigma^2_{2}.
\]
Assume that $\nu_2+\varsigma_1\neq0$. Then:
\begin{itemize}
    \item[I.] If ${h}\geq 0$ and $F(0)\geq 0$, then the stability of the trivial solution $Z(t,\tau)=0$ of \eqref{linear system with one delay}  does not change as $\tau\geq 0$ increases.
    \item[II.] If $h<0$ and $0\leq F(0)< h^2/4$, there exist at least one and at most two purely imaginary characteristic roots 
    $\hat{\lambda}(\tau_{k})=i\hat{\lambda}_{2,k}(\hat{\tau}_k)$ of \eqref{characteristic new general}, with $0<\hat{\lambda}_{2,1}<\hat{\lambda}_{2,2}$ given by
    \[
    \hat{\lambda}_{2,k}=\sqrt{\frac{-{h}+(-1)^{k}\sqrt{{h}^2-4{F(0)}}}{2}}, \quad k=1,2
    \]
    where $\hat{\tau}_k\geq0$ is defined as the smallest solution of the system 
\[
\sin(\hat{\tau}_k \hat{\lambda}_{2,k})=\frac{-\hat{\lambda_{2,k}}\big(\varsigma_2 \hat{\lambda}^2_{2,k}+\nu_1\varsigma_1-\nu_2 \varsigma_2\big)}{(\varsigma_2 \hat{\lambda}_{2,k})^2+\varsigma^2_{1}} \quad \text{and} \quad  \cos(\hat{\tau_{k}} \hat{\lambda}_{2,k})=\frac{(\varsigma_1-\nu_1\varsigma_2)\hat{\lambda}^2_{2,k}-\varsigma_1\nu_2}{(\varsigma_2 \hat{\lambda}_{2,k})^2+\varsigma^2_{1}}, 
\]
in which the stability of the trivial solution $Z(t,\tau)=0$ of \eqref{linear system with one delay}  can change a finite number of times as $\tau$ increases, and eventually becomes unstable.
\item[III.] If $F(0)<0$, $\hat{\lambda}(\tau_{2})=i\hat{\lambda}_{2,2}(\hat{\tau}_{2})$ 
is the only purely imaginary root of \eqref{characteristic new general}. Moreover, an unstable trivial solution $Z(t,\tau)=0$ of \eqref{linear system with one delay},  never becomes stable for any $\tau \geq 0$. Furthermore, if $Z(t,\tau)=0$ is asymptotically stable for $\tau=0$, then it is uniformly asymptotically stable for $\tau<\hat{\tau}_{2}$, and becomes unstable for $\tau>\hat{\tau}_{2}.$
\end{itemize}
\end{theorem}
\begin{proof}[Proof of Theorem \ref{stability theorem}] The proof follows along the same lines as that of Theorem 4.1 in Section 3 of \cite{kuang}. 
\end{proof}
\begin{remark}\label{rem1} Let $\lambda(\tau)=\lambda_1(\tau)+i \lambda_2(\tau)$ be a solution of \eqref{characteristic new general}. Then,
\[
P(\lambda,\tau)=0 \quad \Leftrightarrow \quad p(\lambda(\tau))=-q(\lambda(\tau))e^{-\lambda(\tau)\tau},
\]
For $\tau=0$ we have
\[
p(\lambda(0))=-q(\lambda(0))\quad \Leftrightarrow \quad \lambda^2(0)-(\nu_1+\varsigma_2)\lambda(0)+\nu_2+\varsigma_1=0,
\]
Consequently,
\[
\lambda(0)=\frac{\nu_1+\varsigma_2 \pm \sqrt{(\nu_1+\varsigma_2)^2-4(\nu_2+\varsigma_1)}}{2}.
\]
\begin{itemize}
\item If $\nu_{2}+\varsigma_1<0$ the trivial solution $Z(t,\tau)=0$ is unstable for $\tau=0$. In addition, if $h\geq 0$ and $\nu_{2}-\varsigma_1\leq 0$ (then $F(0)\geq 0$), the trivial solution $Z(t,\tau)=0$ is unstable for all $\tau\geq 0.$
\item If $\nu_{2}+\varsigma_1>0$ and $\nu_{1}+\varsigma_{2}\gtrless 0$ the trivial solution $Z(t,\tau)=0$ is unstable (asymptotically stable) for $\tau=0$. In addition, if $h<0$, $\nu_{2}-\varsigma_{1}\geq 0$ (then $0\leq F(0)$) and $0\leq F(0)<h^2/4$, the linear delayed system \eqref{linear system with one delay} admits at least one and at most two periodic solutions $Z(t,\hat{\tau}_{k})=\exp[{\hat{\lambda}(\hat{\tau}_k) t}]Z_{0}$, with $\hat{\lambda}(\hat{\tau}_{k})=i\hat{\lambda}_{2,k}(\hat{\tau}_{k})$ and $Z_{0}=Z(0,\hat{\tau}_{k})\neq \textbf{0}$ with $0<\hat{\tau}_1<\hat{\tau}_2$. Further, the stability of $Z(t,\tau)=0$ can change as $\tau$ passes through
$\hat{\tau}_1$ and $\hat{\tau}_2$.
\item If $F(0)=\nu^2_{2}-\varsigma^2_1<0$, the linear delayed system \eqref{linear system with one delay} admits a unique periodic solution $Z(t,\hat{\tau}_2)=\exp[{\hat{\lambda}(\hat{\tau}_2) t}]Z_{0}$, with $Z_{0}=Z(0,\hat{\tau}_{2})\neq \textbf{0}$. In addition, if we have
\[
\nu_1+\varsigma_{2}<0, \quad \text{and} \quad \nu_{2}+\varsigma_{1}>0,
\]
the trivial solution $Z(t,\tau)=0$ is uniformly asymptotically stable for $0\leq \tau<\hat{\tau}_{2}$ and unstable for  $\tau>\hat{\tau}_{2}.$
\end{itemize}
\end{remark}

\section{Linear stability in the autonomous case}
\label{sec:linear stab auton}
This section provides some results on the dynamics of the autonomous delayed differential equation system associated with \eqref{main eq 1}, namely, 
\begin{equation}\label{autonomous main eq}
\begin{split}
\dot{N}(t)&=N(t)f(N(t))-\frac{\varphi  N(t)C(t)}{\nu+N(t)} +\sigma N(t)C(t),\\
\dot{C}(t)&=-\eta C(t)-l_{e}C(t) +\frac{\gamma C(t-\tau)N(t-\tau)}{\nu+N(t-\tau)},
\end{split}
\end{equation}
where $l_e(t)\equiv l_e>0$ for all $ t\in \mathbb{R}$.

We begin by studying the existence and stability of the equilibria. First, notice that \eqref{autonomous main eq} can be written as
\begin{equation}\label{new main EQ aut no-delay}
\dot{Y}=H(Y,Y_{\tau})= \begin{pmatrix}
N\big[f(N)+C(\sigma-g_{\varphi}(N))\big]\\
g_{\gamma}(N_{\tau})C_{\tau}N_{\tau}
-(\eta+l_e)C
\end{pmatrix},
\end{equation}
where $Y=Y(t)=(N(t),C(t))^{tr}$, $Y_{\tau}=Y_{\tau}(t)=(N(t-\tau),C(t-\tau))^{tr}$, $f$ satisfies the assumptions stated in Section \ref{sec:intro}, and 
\[
g_s(N)=\frac{s}{\nu+N}.
\]
The equilibria of \eqref{autonomous main eq} are the solutions of the nonlinear system
\[
\begin{split}
N\big[f(N)+C(\sigma-g_{\varphi}(N))\big]&=0, \\
C[g_{\gamma}(N)N-(\eta+l_e)]&=0.
\end{split}
\]
Direct computations yield two principal equilibria:
\[
\begin{split}
Y_{0}&=(0,0), \hspace{0.7 cm} \text{(trivial equilibrium point)}\\
Y_{1}&=(N_{\dagger},0), \hspace{0.42 cm} \text{(the criminal-free equilibrium point).}
\end{split}
\]
where $N_{\dagger}$ is the unique positive solution of $f(N)=0$. For example, when we have a logistic growth function, we have $N_{\dagger}=m$. A further admissible equilibrium, \[
Y_2=(\hat{N},\hat{C}), \hspace{0.42 cm} \text{(coexistence equilibrium point)}
\]
may arise when 
\begin{equation}\label{non trivial equilibria equation}
f(\hat{N})=\hat{C}(g_{\varphi}(\hat{N})-\sigma)\quad \text{and} \quad  g_{\gamma}(\hat{N})\hat{N}=\eta+l_e.
\end{equation}
Direct computations show that in the logistic case, we obtain
\[
\hat{N}=\frac{\nu (\eta+l_e)}{\gamma-(\eta+l_e)} \quad \text{and} \quad \hat{C}=\frac{f(\hat{N})}{g_{\varphi}(\hat{N})-\sigma}=\frac{\mu\gamma \nu(m\gamma-(\eta+l_e)(m+\nu))}{(\varphi(\gamma-(\eta+l_e))-\sigma \gamma \nu)(\gamma-(\eta+l_e))}.
\]


For $\hat{N}$ to be positive, we need $\gamma > \eta+l_e$. This condition reflects the requirement that the maximal criminalization rate be sufficiently large relative to the combined effects of natural removal and law enforcement for a positive coexistence equilibrium to be possible. On the other hand, assuming $f(\hat{N})>0$, positivity of $\hat{C}$ requires $g_{\varphi}(\hat{N})>\sigma$. This condition express the balance between the negative effects of criminal interactions on the non-criminal population and the community-based protective response represented by $\sigma$.

\subsection{Linear stability.} \label{subsec: linear stability} In the following lines, we focus our attention on the study of the linear stability of each equilibrium point $Y_{\ast}=(N_{\ast}, C_{\ast})$ of the autonomous delayed system \eqref{new main EQ aut no-delay}. Therefore, we consider the associated linearized system  
\begin{equation}\label{linearized system with delay}
\dot{Z}=D_{Y}H(Y_{\ast})Z+D_{Y_{\tau}}H(Y_{\ast})Z_{\tau},
\end{equation}
with
\[
\begin{split}
& D_{Y}H(Y_{\ast})=\begin{pmatrix}
f(N_{\ast})+C_{\ast}(\sigma-g_{\varphi}(N_{\ast}))+N_{\ast}[f^{\prime}(N_{\ast})-C_{\ast}g^{\prime}_{\varphi}(N_{\ast})]& N_{\ast}(\sigma-g_{\varphi}(N_{\ast})) \\
0& -(\eta+l_e)
\end{pmatrix}, \\
& D_{Y_{\tau}}H(Y_{\ast})=\begin{pmatrix}
0 & 0\\
[g^{\prime}_{\gamma}(N_{\ast})N_{\ast}+g_{\gamma}(N_{\ast})]C_{\ast}& g_{\gamma}(N_{\ast})N_{\ast}
\end{pmatrix}.
\end{split}
\]

Following the results in Section \ref{preliminar}, if we denote 
\[
A=D_{Y}H(Y_{\ast}) \quad \text{and} \quad  B=D_{Y_{\tau}}H(Y_{\ast}),
\]
the associated characteristic equation of \eqref{linearized system with delay} is given by
\begin{equation}
P(\lambda,\tau)=0,\quad \text{with} \quad P(\lambda,\tau)=\lambda^2-a_{\ast}\lambda+b_{\ast}+(c_{\ast}-d_{\ast} \lambda)e^{-\tau\lambda},
\end{equation}
where 
\[
a_{\ast}:=\text{tr}(A), \quad b_{\ast}:=\text{det}(A), \quad c_{\ast}=\det{(a^{1}|b^2)} + \det{(b^1|a^2)} \quad \text{and} \quad d_{\ast}=\text{tr}B,
\]
with $a^{j}$ and $b^{j}$ denote the $j$-th columns of $A$ and $B$, respectively. Direct computations show that
\[
\begin{split}
c_{\ast}&=N_{\ast}[g_{\gamma}(N_{\ast})(f(N_{\ast})+N_{\ast}f^{\prime}(N_{\ast}))+C_{\ast}N_{\ast}\big(g_{\varphi}(N_{\ast})g^{\prime}_{\gamma}(N_{\ast})-g_{\gamma}(N_{\ast})g^{\prime}_{\varphi}(N_{\ast})-\sigma g^{\prime}_{\gamma}(N_{\ast})\big)],\\
&=N_{\ast}[g_{\gamma}(N_{\ast})(f(N_{\ast})+N_{\ast}f^{\prime}(N_{\ast}))-\sigma C_{\ast}N_{\ast}g^{\prime}_{\gamma}(N_{\ast})],
\end{split}
\]
We are now in position to establish local stability conditions for each each equilibrium by applying Theorem \ref{stability theorem} and Remark \ref{rem1}. 
Consequently, we identify for each equilibrium point
\[
\begin{split}
a_{\ast}&=\nu_1, \quad b_{\ast}=\nu_2,\quad
c_{\ast}=\varsigma_1, \quad d_{\ast}= \varsigma_2 \quad \text{and} \quad h_{\ast}=a^2_{\ast}-2b_{\ast}-d^2_{\ast}.
\end{split}
\]

\noindent
\textbf{Stability of the trivial equilibrium}. For $Y_{0}=(0,0)$, the characteristic equation is
\[
P_{0}(\lambda,\tau)=0, \quad \Leftrightarrow \quad \lambda^2-(f(0)-(\eta+l_e))\lambda-f(0)(\eta+l_e)=0.
\]
Since $-f(0)(\eta+l_e)<0$ it follows directly that $Y_0$ is unstable for every $\tau\geq 0$ (a saddle-type equilibrium). Notice that the same conclusion follows from Remark \ref{rem1} since in this case
\[
\nu_{1}=f(0)-(\eta+l_e), \quad \nu_2=-f(0)(\eta+l_e), \quad \varsigma_1=\varsigma_2=0,
\]
then
\[
\nu_{2}+\varsigma_1=-f(0)(\eta+l_e)<0, \quad h=f^{2}(0)+(\eta+l_e)^{2}>0 \quad \text{and} \quad F(0)=\nu^2_2>0.
\]

\noindent
\textbf{Stability of the criminal-free equilibrium}. We next study the stability of the criminal-free equilibrium $Y_{1}=(N_{\dagger},0)$ and the ocurrence of purely imaginary characteristic roots in the corresponding linearized system.
Firstly, the characteristic equation is given by
\[
P_{1}(\lambda,\tau)=0, \quad \Leftrightarrow \quad \lambda^2-a_{\dagger}\lambda+ b_{\dagger}+(c_{\dagger}-d_{\dagger}\lambda)e^{-\tau \lambda}=0,
\]
where
\[
\begin{split}
a_{\dagger}&=N_{\dagger}f^{\prime}(N_\dagger)-(\eta+l_e), \hspace{0.5 cm}  c_{\dagger}=N^2_{\dagger}g_{\gamma}(N_{\dagger})f^{\prime}(N_\dagger),\\
b_{\dagger}&=-N_{\dagger}f^{\prime}(N_\dagger)(\eta+l_e),\hspace{0.6 cm} d_{\dagger}=g_{\gamma}(N_{\dagger})N_\dagger.
\end{split}
\]
Since $f^{\prime}(N_{\dagger})<0$, we deduce $a_{\dagger}<0$, $c_{\dagger}<0$, $b_{\dagger}>0$, and $d_{\dagger}>0.$ For the next proposition, we require the following quantities.
\begin{equation}\label{cantidades para el equilibria libre de tumor}
\begin{split}
b_{\dagger}+c_{\dagger}&=\frac{N_{\dagger}f^{\prime}(N_{\dagger})}
{\nu+N_{\dagger}}\big(N_{\dagger}(\gamma-(\eta+l_e))-\nu(\eta+l_e)\big),\\
b_{\dagger}-c_{\dagger}&=-\frac{N_{\dagger}f^{\prime}(N_{\dagger})}{\nu+N_{\dagger}}\big(N_{\dagger}(\gamma+(\eta+l_e))+\nu(\eta+l_e)\big),\\
a_{\dagger}+d_{\dagger}&=N_{\dagger}f^{\prime}(N_{\dagger})+\frac{N_{\dagger}(\gamma-(\eta+l_e))-\nu(\eta+l_e)}{\nu+N_{\dagger}},\\
h_{\dagger}&=\big(N_{\dagger}(f^{\prime}(N_{\dagger})\big)^2-\frac{\big(N_{\dagger}(\gamma-(\eta+l_e))-\nu(\eta+l_e)\big)\big(N_{\dagger}(\gamma+(\eta+l_e))+\nu(\eta+l_e)\big)}{(\nu+N_{\dagger})^2}.
\end{split}
\end{equation}
\begin{prop}\label{prop auton criminal-free stability}
Assume that
\[
N_{\dagger}(\gamma-(\eta+l_e))\neq \nu(\eta +l_e).
\]
\begin{itemize}
\item[a.] If $\displaystyle{N_{\dagger}(\gamma-(\eta+l_e))>\nu(\eta +l_e)}$, then  $Y_{1}=(N_{\dagger},0)$ is unstable for all $\tau\geq 0$.
\item[b.] If $\displaystyle{N_{\dagger}(\gamma-(\eta+l_e))<\nu(\eta +l_e)}$ then  $Y_{1}=(N_{\dagger},0)$ is asymp\-totically stable for all $\tau\geq 0$. 
\end{itemize}
\end{prop}
\begin{proof}
The associated function $F(s)$ given by \eqref{polynomial function} satisfies  
\[
F(0)=b^2_{\dagger}-c^2_{\dagger}=(b_{\dagger}-c_{\dagger})(b_{\dagger}+c_{\dagger}).
\]
From the hypothesis and \eqref{cantidades para el equilibria libre de tumor} we deduce that $\displaystyle{F(0)\neq 0}$ and  $\text{sgn}(F(0))=\text{sgn}(b_{\dagger}+c_{\dagger})$. Now, in case a., we have 
\[
b_{\dagger}+c_{\dagger}<0 \quad \text{and} \quad F(0)<0.
\]
and in case $b.$, we have
\[
a_{\dagger}+d_{\dagger}<0, \quad b_{\dagger}+c_{\dagger}>0, \quad F(0)>0 \quad  \text{and} \quad h>0.
\]
The conclusion follows directly from Theorem \ref{stability theorem} case I and case III.
\end{proof}
\begin{remark}
For case a in Proposition \ref{prop auton criminal-free stability}, the linearized system \eqref{linearized system with delay} has a unique periodic solution $\tilde{Y}(t,\tilde{\tau})=\exp[\tilde{\lambda}(\tilde{\tau}_2)]\tilde{Y}_{0}$ with $\tilde{Y}_{0}=\tilde{Y}(0,\tilde{\tau}_2)\neq \textbf{0}$ and $\tilde{\lambda}(\tilde{\tau}_2)=i\tilde{\lambda}_{2,2}(\tilde{\tau}_2)$ where
\[
    \tilde{\lambda}_{2,2}=\sqrt{\frac{-h_{\dagger}+\sqrt{h_{\dagger}^2-4(b^2_{\dagger}-c^2_{\dagger})}}{2}},
    \]
    with $0\leq \tilde{\tau}_{2}$ given as the first solution of the system 
\[
\sin(\tilde{\tau}_2 \tilde{\lambda}_{2,2})=\frac{-\tilde{\lambda}_{2,2}\big(d_{\dagger} \tilde{\lambda}^2_{2,2}+a_{\dagger}c_{\dagger}-b_{\dagger} d_{\dagger}\big)}{(d_{\dagger} \tilde{\lambda}_{2,2})^2+c^{2}_{\dagger}} \quad \text{and} \quad  \cos(\tilde{\tau}_{2} \hat{\lambda}_{2,2})=\frac{(c_{\dagger}-a_{\dagger}d_{\dagger})\tilde{\lambda}^2_{2,2}-c_{\dagger}b_{\dagger}}{(d_{\dagger} \tilde{\lambda}_{2,2})^2+c^{2}_{\dagger}}. 
\]
The occurrence of purely imaginary characteristic roots at this critical delay suggests a possible stability change in the nonlinear system. Establishing a Hopf bifurcation, however, would require additional non-degeneracy and transversality conditions.
\end{remark}

\textbf{Stability of the coexistence equilibrium}. Recall that the coexistence equilibrium $Y_2=(\hat{N},\hat{C})$ satisfies
\[
g_{\gamma}(\hat{N})\hat{N}=\eta+l_e \quad \Leftrightarrow \quad \hat{N}=\frac{\nu(\eta+l_e)}{\gamma-(\eta+l_e)}, \quad \text{for} \quad \gamma>(\eta+l_e),
\]
in addition with
\[
f(\hat{N})=\hat{C}(g_{\varphi}(\hat{N})-\sigma) \quad \Leftrightarrow \quad \nu f(\hat{N})=\left(\frac{\varphi}{\gamma}(\gamma-(\eta+l_e))-\sigma \nu\right)\hat{C}.
\]
Consequently, the coexistence equilibrium $Y_2=(\hat{N},\hat{C})$ is admissible, i.e., $\hat{N},\hat{C}\geq 0$, in the following cases:
\begin{enumerate}
\item[\textit{Case }$I.$ ]
    \[
\text{If} \quad \hat{N}<N_{\dagger} \quad \Leftrightarrow \quad N_{\dagger}(\gamma-(\eta+l_e))>\nu(\eta+l_e) \quad \text{and} \quad \frac{\varphi}{\gamma}> \frac{\sigma \nu}{\gamma-(\eta+l_e)}.
    \]
\item[\textit{Case }$II.$]
    \[
\text{If} \quad \hat{N}>N_{\dagger} \quad \Leftrightarrow \quad N_{\dagger}(\gamma-(\eta+l_e))<\nu(\eta+l_e)\quad \text{and} \quad \frac{\varphi}{\gamma}< \frac{\sigma \nu}{\gamma-(\eta+l_e)}.
    \]
\end{enumerate}

Concerning the linear stability of $ Y_{2}= (\hat{N},\hat{C})$, we proceed as before. First, the corresponding characteristic equation is given by
\[
P_{2}(\lambda,\tau)=0, \quad \Leftrightarrow \quad \lambda^2-\hat{a}\lambda+ \hat{b}+(\hat{c}-\hat{d}\lambda)e^{-\tau \lambda}=0,
\]
where the coefficients are calculated taking into account equation \eqref{non trivial equilibria equation}, and given explicitly by
\begin{equation}\label{parameters for coexistence equilibrium}
\begin{aligned}
\hat{a}&=\hat{N} [f^{\prime}(\hat{N})-\hat{C}g^{\prime}_{\varphi}(\hat{N})]-(\eta+l_e)\\
\hat{b}&=-(\eta+l_e)\hat{N} [f^{\prime}(\hat{N})-\hat{C}g^{\prime}_{\varphi}(\hat{N})]
\end{aligned}
\qquad\text{and}\qquad
\begin{aligned}
\hat{c}&=\frac{\gamma \nu\hat{N}f(\hat{N})}{(\nu+\hat{N})^2}-\hat{b},\\
\hat{d}&=\eta+l_e.
\end{aligned}
\end{equation}
In addition, and for future purposes, we compute  
\begin{equation}\label{signo de parametros}
\begin{split}
\hat{b}+\hat{c}&=\frac{\gamma \nu\hat{N}f(\hat{N})}{(\nu+\hat{N})^2}=\frac{\gamma \nu\hat{N}\hat{C}\left(g_{\varphi}(\hat{N})-\sigma\right)}{(\nu+\hat{N})^2}=\frac{\gamma \hat{N}\hat{C}}{(\nu+\hat{N})^2}\left(\frac{\varphi}{\gamma}(\gamma-(\eta+l_e))-\sigma \nu\right),\\
\hat{b}-\hat{c}&=-2(\eta+l_e)\hat{N}\left(f^{\prime}(\hat{N})-\hat{C}g^{\prime}_{\varphi}(\hat{N})-\frac{\gamma\nu f(\hat{N})}{2(\eta+l_e)(\nu+\hat{N})^2}\right),\\
&=-(\eta+l_e)\left(2\hat{N}f^{\prime}(\hat{N})+\frac{\hat{C}}{\nu+\hat{N}}\left(\frac{\varphi}{\gamma}(\gamma+\eta+l_e)-\sigma\nu\right)\right).\\
\hat{a}+\hat{d}&=\hat{N} [f^{\prime}(\hat{N})-\hat{C}g^{\prime}_{\varphi}(\hat{N})],\\
\hat{h}&=\left(\hat{N}[f^{\prime}(\hat{N})-\hat{C}g^{\prime}_{\varphi}(\hat{N})]\right)^2.
\end{split}
\end{equation}
We are now ready to establish stability criteria for the coexistence equilibrium.
\begin{prop} 
\label{prop: estab auton coexist}
Let us consider the parameters $\hat{a},\hat{b}, \hat{c}$ and $\hat{d}$ given by \eqref{parameters for coexistence equilibrium}.
\begin{itemize}
    \item[a.] Assume that $\displaystyle{ f^{\prime}(\hat{N})<\hat{C}g^{\prime}_{\varphi}(\hat{N})}$ and \textit{Case I} hold; that is,
\[
f^{\prime}(\hat{N})<-\frac{\hat{C}\varphi}{(\nu+\hat{N})^2}, \quad N_{\dagger}(\gamma-(\eta+l_e))>\nu(\eta+l_e) \quad \text{and} \quad \frac{\varphi}{\gamma}> \frac{\sigma \nu}{\gamma-(\eta+l_e)}.
\]
Then, the coexistence equilibrium $Y_2=(\hat{N},\hat{C})$ is locally asymptotically stable, for all sufficiently small $\tau\geq 0$. \\
\item[b.] Assume that \textit{Case II} holds; that is,
\[
N_{\dagger}(\gamma-(\eta+l_e))<\nu(\eta+l_e) \quad \text{and} \quad \frac{\varphi}{\gamma}< \frac{\sigma \nu}{\gamma-(\eta+l_e)}, 
\]
then the coexistence equilibrium $Y_2=(\hat{N},\hat{C})$ is unstable for all sufficiently small $\tau\geq 0$. Moreover, if 
\[
2\hat{N}f^{\prime}(\hat{N})+\frac{\hat{C}}{\nu+\hat{N}}(\frac{\varphi}{\gamma}(\gamma+\eta+l_e)-\sigma\nu)>0, \qquad (\triangle)
\]
then $Y_{2}$ is unstable for all $\tau\geq 0.$\\
\item[c.] Assume that \textit{Case I} and $(\triangle)$ hold. If $f^{\prime}(\hat{N})<\hat{C}g^{\prime}_{\varphi}(\hat{N})$ also holds, then, $Y_2$ is uniformly asymptotically stable for $0\leq \tau<\hat{\tau}_{2}$ and unstable for  $\tau>\hat{\tau}_{2}$, where
\[
    \hat{\lambda}_{2,2}=\sqrt{\frac{-\hat{h}+\sqrt{\hat{h}^2-4(\hat{b}^2-\hat{c}^2)}}{2}},
    \]
    with $0\leq \hat{\tau}_{2}$ given as the smallest positive solution
\[
\sin(\hat{\tau}_2 \hat{\lambda}_{2,2})=\frac{-\hat{\lambda}_{2,2}\big(\hat{d} \hat{\lambda}^2_{2,2}+\hat{a}\hat{c}-\hat{b} \hat{d}\big)}{(\hat{d} \hat{\lambda}_{2,2})^2+\hat{c}^{2}} \quad \text{and} \quad  \cos(\hat{\tau}_2 \hat{\lambda}_{2,2})=\frac{(\hat{c}-\hat{a}\hat{d})\hat{\lambda}^2_{2,2}-\hat{c}\hat{b}}{(\hat{d} \hat{\lambda}_{2,2})^2+\hat{c}^{2}}. 
\]
\end{itemize}
\end{prop}
\begin{proof}
As we did in the proof of Proposition \ref{prop auton criminal-free stability}, the corresponding function $F(s)$ given by \eqref{polynomial function}, in this case satisfies  
\[
F(0)=\hat{b}^2-\hat{c}^2=(\hat{b}-\hat{c})(\hat{b}+\hat{c}).
\] 
Let us prove the case $a.$  Then, from \eqref{signo de parametros} we deduce that
\[
\hat{b}+\hat{c}>0 \quad \text{and} \quad \hat{a}+\hat{d}\lessgtr0.
\]
The conclusion follows by Remark \ref{rem1}. In case b., we have $\displaystyle{\hat{b}-\hat{c}<0}$, then by Remark \ref{rem1} we deduce that $Y_2$ is unstable for $\tau\geq 0$ and small. Now, if $(\triangle)$ also holds, then from \eqref{signo de parametros} we have $\hat{b}+\hat{c}<0$, so $F(0)>0$. Since $h\geq 0$, from Theorem \ref{stability theorem} case I, (or by Remark \ref{rem1}) that $Y_2$ is unstable for all $\tau\geq 0.$ The proof of case c. follows analogously by applying case III of Theorem \ref{stability theorem}.
\end{proof}
\begin{remark}
For case c in Proposition \ref{prop: estab auton coexist}, the linearized system \eqref{linearized system with delay} has a periodic solution 
 $\hat{Y}(t,\hat{\tau})=\exp[\hat{\lambda}(\hat{\tau}_2)]\hat{Y}_{0}$ with $\hat{Y}_{0}=\hat{Y}(0,\hat{\tau}_2)\neq \textbf{0}$ and $\hat{\lambda}(\hat{\tau}_2)=i\hat{\lambda}_{2,2}(\hat{\tau}_2)$, with $\hat{\tau}_2$ and $\hat{\lambda}_{2,2}$ as in Proposition \ref{prop: estab auton coexist}. Thus, the critical value $\hat{\tau}_2$ marks a transition in the local stability of the coexistence equilibrium: it is asymptotically stable for $0\le \tau \le \hat{\tau}_2$ and unstable for $\tau > \hat{\tau}_2$. This highlights the role of the delay as a mechanism capable of destabilizing coexistence between the criminal and non-criminal populations.
\end{remark}
We conclude this section by observing that Theorem \ref{stability theorem}
provides explicit conditions under which purely imaginary characteristic roots
occur at the criminal-free equilibrium $Y_1=(N_\dagger,0)$ and at the
coexistence equilibrium $Y_2=(\hat N,\hat C)$. At the corresponding critical
delay values, the linearized system \eqref{linearized system with delay}
admits nontrivial periodic solutions associated with the purely imaginary
characteristic roots. These results identify candidate critical delays for
stability changes of the nonlinear autonomous system \eqref{autonomous main eq}.
In particular, the occurrence of a simple pair of purely imaginary roots,
together with the appropriate transversality condition, would provide the
spectral ingredients required for a Hopf bifurcation analysis. This motivates
the study of periodic solutions of the nonlinear system \eqref{main eq 1},
which we undertake in the next section.

\section{Non-autonomous case: existence of periodic solutions}
\label{sec: non auton periodic sols}

 In this section, we  consider the original non-autonomous system
\eqref{main eq 1}, in which it shall be assumed that the nonnegative law enforcement function $l_e$ is continuous and $T$-periodic. In particular, we  focus on the existence of strictly positive $T$-periodic solutions. 

We first observe that, when the amplitude of $l_e$ is small, an argument based on the implicit function theorem can be established to obtain a $T$-periodic solution near a positive equilibrium 
of the corresponding autonomous system. This approach is is analogous to the 
one employed in \cite{AAR2025}. However, since law-enforcement intensity may exhibit large temporal variations, we seek sufficient conditions depending only on its average value, provided that the delay is sufficiently small. We observe that if the amplitude of $l_e$ is small, 
then an argument based on the implicit function theorem can be established to obtain a $T$-periodic solution near a positive equilibrium 
of an autonomous system. 
This procedure is analogous to the 
one employed in \cite{AAR2025}. 
However, we emphasize the fact that, due to different factors, the law enforcement function may have oscillations of large amplitude; consequently, we are interested in finding sufficient conditions that depend only 
on the average of $l_e$ provided that the delay is small enough.

Throughout this section, we use the term \textit{counteroffensive measures} as a shorthand for the community-based protective response represented by $\sigma$, and introduced previously.
 
With this in mind, we shall consider two different scenarios: on the one hand, we shall assume that the counteroffensive measures expressed by $\sigma$ are large enough to compensate the effect of the large oscillations in $l_e$; on the other hand, when $\sigma$ is small, it shall be shown that 
$T$-periodic solutions exist as a natural extension of the coexistence equilibrium in the autonomous case. 
In this latter situation, it may be understood 
that, if the average of $l_e$ is large, 
then the criminal population tends to extinction, and no periodic solutions exist while, otherwise, the criminal and non-criminal populations will coexist periodically.

We shall make use of the Leray-Schauder topological degree. 
We remark that, when dealing with periodic problems, typically two different approaches may be implemented. The first one consists in finding fixed points of the Poincaré operator, defined in an appropriate open subset of the phase space, while the second one transforms the problem in a functional equation in a space of periodic functions. Here, our aim is to deduce the existence of solutions for small values of $\tau$ by analyzing the undelayed case; thus, the first approach is not convenient due to the fact that the phase space varies with $\tau$. We shall work, instead, 
in $C_T:=C_T(\R,\R)$, namely
 the Banach space of continuous $T$-periodic functions. For convenience, let us denote the average of an element $\theta\in C_T$ by $\overline\theta$, that is, $\overline\theta:=\displaystyle \frac 1T\int_0^T\theta(s)\, ds$.

Our results will be based on standard continuation theorems (see e.g. \cite[Thm. 6.3]{Amster}). For the reader's convenience, a sketch of the method is introduced as follows. Consider a system 
\begin{equation}\label{abstract}
    X'(t)=F(t,X(t),X(t-\tau)),
\end{equation}
where $F:\R\times \R^2\to \R^2$ is continuous and $T$-periodic in its first coordinate and define the Nemitskii operator $N:(C_T)^2 \times [0,+\infty)\to (C_T)^2$ by $N(X,\tau)(t):= F(t,X(t),X(t-\tau))$.
Next, observe that the problem of finding $T$-periodic solutions of (\ref{abstract})
can be 
 written equivalently as $X=\mathcal K(X,\tau)$, where the operator $\mathcal K:(C_T)^2\times [0,+\infty)\to (C_T)^2$ is given by 
 $\mathcal K(X,\tau):=X(0) + 
\overline{N(X,\tau)} + K(X,\tau)$, with
$$K(X,\tau)(t) :=   \int_0^t [N(X,\tau)(s)- \overline{N(X,\tau)}]\, ds.
$$
A straightforward application of the Arzel\`a-Ascoli theorem shows that $\mathcal K$ is compact; moreover, the following lemma will be useful for the search of solutions when $\tau$ is small:
\begin{lemma}\label{small tau}
    Let $\Omega\subset (C_T)^2$ be open and bounded such that $X\ne \mathcal K(X,0)$ for all $X\in \partial \Omega$. Then there exists $\tau^*>0$ such that, for all $\tau\in (0,\tau^*)$, it is verified that $X\ne \mathcal K(X,\tau)$ for all $X\in\partial \Omega$ and
    $$\deg(I-\mathcal K(\cdot,\tau),\Omega,0)=\deg(I-\mathcal K(\cdot,0),\Omega,0). $$
\end{lemma}
\begin{proof}
    Suppose for contradiction the existence of $\tau_n\searrow 0$ and $X_n\in\partial \Omega$ such that $X_n=\mathcal K(X_n,\tau_n)$. Passing to a subsequence, we may assume that $\mathcal K(X_n,\tau_n)$ converges to some $X$ as $n\to\infty$, so $X_n\to X$ as well. This implies that $X\in\partial \Omega$ and $X=\mathcal K(X,0)$, a contradiction. The conclusion now follows from the homotopy invariance of the Leray-Schauder degree.
\end{proof}
Thus, we will restrict ourselves to study the undelayed case, and $T$-periodic solutions for small values of $\tau$ will be obtained. 
To this end,  
consider the homotopy
$$\mathcal H_\varepsilon (X):= 
X(0) + \overline{N(X,0)} + \varepsilon K(X,0).
$$
As before, for $\varepsilon \in (0,1]$ it is seen that $X$ is a fixed point of $\mathcal H_\varepsilon$
if and only if 
\begin{equation}
    \label{homot}
    X'(t)=\varepsilon F(t,X(t),X(t)),
\end{equation}
and, for $\varepsilon =0$ the topological degree of the map $I-\mathcal H_0$ over some bounded open region $\Omega\subset C_T\times C_T$ is easy to compute as the Brouwer degree of the mapping 
\begin{equation}\label{phi}
\Phi(X):= \overline {F(\cdot,X,X)}\qquad X\in \R^2,
\end{equation}
over $\Omega\cap\R^2$ where, for simplicity, we identify the subset of constant functions of $C_T$ with $\R$. From the homotopy invariance property, the degree of $I-H_\varepsilon$ over $\Omega$ is constant, provided that $I-\mathcal H_\varepsilon$ does not vanish on the boundary of $\Omega$.  


Taking into account the previous setting, it is necessary to verify that the following hypotheses are met: \textbf{(I)} Problem  \eqref{homot} has no solutions in $\partial\Omega$ for $0<\varepsilon \le 1$, \textbf{(II)} $\Phi$ does not vanish on $\partial \Omega\cap\R^2$ and \textbf{(III)} The Brouwer degree of $\Phi$ in $\Omega\cap\R^2$ is not zero. Combined with Lemma \ref{small tau}, this will suffice to guarantee the existence of $T$-periodic solutions for small values of $\tau$.


Let   us first consider the  change of  variables
$$u(t):=\ln N(t), \qquad v(t):= \ln C(t),$$
so the original problem is equivalent  to that of finding a solution $(u,v)\in C_T\times C_T$ of the system  
$$\left\{\begin{array}{l}
 u'(t)= f(e^{u(t)}) + \left(\sigma - \frac \varphi{\nu + e^{u(t)}}\right) e^{v(t)},      \\
v'(t)=-(\eta +l_{e}(t)) +\frac{\gamma  e^{u(t)}}{\nu +e^{u(t)}}e^{v(t-\tau)-v(t)}.     
\end{array}
\right.$$
As we mentioned earlier, we can set $\tau=0$ and consider the one-parameter problem family
\begin{equation}
    \label{homot-main}
    \left\{\begin{array}{l}
 u'(t)= \varepsilon\left( f(e^{u(t)}) + \left(\sigma - \frac \varphi{\nu + e^{u(t)}}\right) e^{v(t)}\right),      \\
v'(t)=\varepsilon\left(-(\eta +l_{e}(t)) +\frac{\gamma  e^{u(t)}}{\nu +e^{u(t)}}\right)   
\end{array}
\right.
\end{equation}
with $\varepsilon\in (0,1]$ and the mapping $\Phi:\R^2\to \R^2$ given by
\begin{equation}
    \label{phi-main}
    \Phi(u,v):= \left( f(e^u) + \left(\sigma - \frac\varphi{\nu + e^u}\right)e^v, 
\frac{\gamma e^u}{\nu+e^u} - (\eta+\overline {l_e})\right).
\end{equation}

Prior to the analysis of the two proposed scenarios, let us observe that, taking average in the second equation of (\ref{homot-main}), it is verified that 
$$\eta + \overline{l_e} = \frac\gamma T\int_0^T \frac{ e^{u(t)}}{\nu +e^{u(t)}}\, dt<\gamma,
$$
thus, a necessary condition for the existence of a $T$-periodic solution is
\begin{equation}
    \label{cond-nec-siempre}  
  \gamma > \eta+\overline {l_e}. 
\end{equation}
However, as  shown below, this condition is not sufficient and 
some extra assumptions shall be needed according to the different cases.

\subsection{With great oscillations come great counteroffensive measures}

Here, we assume that the counteroffensive measures are strong, i.e., $\sigma$ is large. More precisely,
\begin{equation}\label{contraof}
\sigma >\frac\varphi\nu.
\end{equation}
Under this assumption, a second necessary condition for the existence of $T$-periodic solutions in the undelayed case reads 
\begin{equation}
    \label{cond-nec}  
  \eta+\overline {l_e}>\frac {\gamma N_\dagger}{\nu + N_\dagger},  
\end{equation}
where $N_\dagger$ is the unique zero of the function  $f$.
Indeed, let $(u,v)$ be a $T$-periodic solution of \eqref{homot-main}
and assume that  $u$ achieves its absolute minimum $u_{\min}$ at some value $t_{\min}$, then its  derivative vanishes, whence 
$$f(e^{u_{\min}}) = -\left(\sigma - \frac \varphi{\nu + e^{u_{\min}}}\right) e^{v(t_{\min})}< 
-\left(\sigma - \frac \varphi{\nu }\right) e^{v(t_{\min})}
<0, 
$$
that is,  $e^{u_{\min}} >N_\dagger$. 
Next, integrating as before the second equation and taking into account that the mapping $x\mapsto \dfrac x{\nu + x}$ is increasing, it is seen that
$$\eta + \overline{l_e} = \frac\gamma T\int_0^T \frac{ e^{u(t)}}{\nu +e^{u(t)}}\, dt > 
\frac\gamma T\int_0^T \frac{ N_\dagger}
{\nu +N_\dagger}\, dt =  \frac{\gamma N_\dagger}
{\nu +N_\dagger}.
$$

Making use of the continuation method described above, 
we shall prove that the necessary conditions 
\eqref{cond-nec-siempre} and (\ref{cond-nec}) are also sufficient when $\tau$ is small. 

It is relevant to observe that, according to the previous computations, any possible  positive $T$-periodic solution $(N,C)$ of the original problem satisfies ${f(N(t))} <0$ for all $t$. For example, if we assume the logistic case, that is, $f(N)=(1-N/K)$ this implies the periodic solution satisfies $N(t)>K$, meaning that $N(t)$ is  above the carrying capacity $K$. This behavior is consistent with the assumption of strong counteroffensive measures: when $\sigma$ is sufficiently large, the effect of the criminal population alone is not sufficient to reduce the non-criminal population below its carrying capacity.

\begin{theorem} \label{teo sol per 1}
Assume that conditions 
\eqref{cond-nec-siempre}, \eqref{contraof} and \eqref{cond-nec} are satisfied.  
Further,  assume that  $f(x)\to -\infty$ as $x\to+\infty$ . 
Then there  exists $\tau^*>0$ such that (\ref{main eq 1}) admits at least one positive $T$-periodic solution for $\tau<\tau^*$. 
\end{theorem}

\begin{proof}
We shall verify the conditions of the continuation theorem for $\tau=0$ with 
\[
\Omega:=\{ (u,v)\in (C_T)^2: r<u(t)<R, \, -M<v(t)<M\},
\]
for some appropriate  $R>r>\ln N_\dagger$ and $M>0$.
Firstly, assume that  $u,v\in C^1_T:=C_T\cap C^1(\R)$ satisfy \eqref{homot-main} for some $\varepsilon\in (0,1]$, then we already know that  $e^{u(t)} > N_\dagger$  for all $t$. 
Moreover, from the second equation we obtain:
$$\frac 1T \int_0^T \frac{ e^{u(t)}}{\nu +e^{u(t)}}\, dt = \frac{\eta+\overline{l_e}}\gamma :=\Lambda<1,
$$
and the mean value theorem implies, for some $t$, that 
$\displaystyle \frac{ e^{u(t)}}{\nu +e^{u(t)}}=\Lambda$. 
In consequence,
$$e^{u(t)} = \frac {\Lambda\nu}{1-\Lambda}:=K,$$
which, in turn, implies
$$e^{u_{\min}}\le K\le e^{u_{\max}}.$$

Next, employing again  the equality $u'(t_{\min})=0$, we deduce: 
$$ e^{v_{\min}}\le e^{v(t_{\min})}= -\frac{f(e^{u_{\min}})}{\sigma - \frac \varphi{\nu + e^{u_{\min}}}}\le -\frac{f(K)}{\sigma - \frac \varphi{\nu + N_\dagger}}.$$
Moreover, the fact that $v'(t)<\gamma$ implies, for all $t$, that  
$v(t)\le v_{\min}+ T\gamma$ is bounded by a  constant $M_0$  depending only on  $f$ and the parameters of the system. On the other hand, if  $u$ achieves its maximum value $u_{\max}$ at some $t_{\max}$,  we also obtain:  
$$ e^{v(t_{\max})}= -\frac{f(e^{u_{\max}})}{\sigma - \frac \varphi{\nu + e^{u_{\max}}}} \ge -\frac{f(K)}{\sigma},
$$
and the fact that  $v_{\max} \le v_{\min} +\gamma T$  allows to find a uniform lower bound $m_0$ for $v$.
 Furthermore, a uniform upper bound for $u$ is obtained from the equality $u'(t_{\max})=0$, namely
\[
(\ast) \quad f(e^{u_{\max}}) = -
\left({\sigma - \frac \varphi{\nu + e^{u_{\max}}}}\right)e^{v(t_{\max})}
> - \sigma e^{M_0}.
\]
Next, we shall adjust the lower bound for $u$ as follows:  since we already know that   $v\ge m_0$, it is seen that  
$$f(e^{u_{\min}})\le -  \left(\sigma - \frac \varphi{\nu + e^{u_{\min}}}\right) e^{m_0} < -  \left(\sigma - \frac \varphi{\nu+N_\dagger}\right) e^{m_0}.$$

Thus, we may fix $r> \ln N_\dagger$ such that $u_{\min}>r$. 
Moreover, the inequality $e^{u_{\min}}\le K$ also implies $\displaystyle \frac {\gamma e^r}{\nu+e^r}<\eta+\overline{l_e}$. 
According to the continuation theorem, we shall set $M>\max\{M_0, -m_0\}$  large enough and compute the Brouwer degree of the mapping 
$\Phi:[r,R]\times [-M,M]\to \R^2$ defined by \eqref{phi-main}. To this end, notice that 
if $u=R\gg 0$ then $\Phi_2(u,v) >0$ for arbitrary $v$ and, in the same way, if 
$u=r$ then $\Phi_2(u,v) <0$ for all values of  $v$. Because 
$-f(e^u)\le -f(e^R)$, it is clear that $\Phi_1(u,M)>0$ for $u\le R$ for $M\gg 0$ and, finally, enlarging $M$ if necessary 
it is also seen that $\Phi_1(u,-M) <0$ for $u\ge r$. 
This proves that 
$$\deg(\Phi, (r,R)\times (-M,M),0)=-1,$$
and so completes the proof. 
\end{proof}
\begin{remark}
Notice that from $(\ast)$ the hypothesis $f\to -\infty$ can be relaxed, it is enough to assume that $f(N)=-\sigma e^{M_0}$ for some $N$, where the value $M_0$ obtained in the previous proof can be expressed in terms of  $N_{\dagger}$ and the parameters of the model. 
\end{remark}

\subsection{Periodic coexistence under weak counteroffensive measures}

In this section, we assume that counteroffensive measures represented by $\sigma$ are weak. Specifically, recall that the condition \eqref{cond-nec-siempre}, which is necessary for the non-delayed case, implies  that any possible $T$-periodic solution of system \eqref{homot-main} satisfies 
$$e^{u_{\min}}\le K \le e^{u_{\max}},$$
where the value $K$ was defined in the preceding sub-section. 
Taking this as a starting point, here we shall assume that 
\begin{equation}
   \label{sigma-chico} \sigma < \frac\varphi{\nu + K}, 
\end{equation}
which, in terms of the original parameters, can be written as
$$\sigma < \varphi \frac{\gamma -(\eta+\overline{l_e})}{\gamma\nu}.
$$
Let us also notice that  (\ref{sigma-chico}) implies that the following condition is necessary: 
\begin{equation}
    \label{K chico} K < N_\dagger.
\end{equation}
Indeed, suppose by contradiction that $K\geq N_{\dagger}$, then we would have 
$$\sigma < \frac\varphi{\nu + K} \le \frac\varphi{\nu + N_\dagger} \quad \Rightarrow \quad \sigma-\frac{\varphi}{\nu+N_{\dagger}<0}.$$

As in the previous section, we may take a value $t_{\min}$ such that $e^{u(t_{\min})} = e^{u_{\min}} \le K$ and then
$$0=u'(t_{\min}) \quad \text{implies that} \quad f(e^{u_{\min}})=- \left(\sigma - \frac \varphi{\nu + e^{u_{\min}}}\right) e^{v(t_{\min})}>0,$$
consequently $e^{u_{\min}}<N_{\dagger}$. 
On the other hand, $\displaystyle{e^{u_{\max}}\geq K\geq N_\dagger}$, so 
\[
e^{u_{\min}}<N_{\dagger}\leq e^{u_{\max}}, \quad (\ast)
\]
therefore, by continuity, there exists at least one $t_{\ast}$ such that
$e^{u(t_{\ast})}=N_\dagger$. At any such point,
\[
u^{\prime}(t_{\ast})=\varepsilon\left(\sigma - \frac \varphi{\nu + e^{u(t_{\ast})}}\right) e^{v(t_{\ast})}<0,
\]
thus every crossing of the level $N_{\dagger}$ is necessarily downward. But this is is contradiction with the periodicity of $u(t)$, since $(\ast)$ requires the periodic trajectory cross the level $N_{\dagger}$ in both direction. Hence $K<N_{\dagger.}$ As in the previous case, we shall prove that the necessary conditions are also sufficient, namely:
\begin{theorem}\label{teo sol per 2}
Assume that (\ref{cond-nec-siempre}), (\ref{sigma-chico}) and (\ref{K chico}) hold. Then there exists $\tau^*>0$ such that problem \eqref{main eq 1} admits at least one positive $T$-periodic solution when $\tau<\tau^*$. 
\end{theorem}

\begin{proof}
As before, we shall firstly find uniform bounds for the possible solutions  $(u,v)$ of \eqref{homot-main} with $\varepsilon \in (0,1]$. We already know that $e^{u_{\min}}\le K\le e^{u_{\max}}$ and 
$$e^{v(t_{\min})} = \frac{f(e^{u_{\min}})}{\frac{\varphi}{\nu + e^{u_{\min}}}-\sigma } \le \frac{f(0) }{\frac{\varphi}{\nu + K}-\sigma },   
$$
which, combined with the boundedness of $v'$, provides an upper bound for  $v$. In turn, this implies that the derivative of $u$ is bounded from above by some constant $A$, whence
$$u_{\max} \le u_{\min} + AT, 
$$    
and consequently
$$K\le e^{u_{\max}}\le e^{u_{\min}} e^{AT} \le Ke^{AT}. $$
Summarizing, there exists $R$ depending only on the parameters such that $|u(t)|<R$ for all $t$. 
Finally, the fact that  $e^{u_{\min}}\le K < N_\dagger$ yields the conclusion
$$e^{v(t_{\min})} = \frac{f(e^{u_{\min}})}{\frac{\varphi}{\nu + e^{u_{\min}}}-\sigma } \ge \frac{f(K)}{\frac \varphi\nu - \sigma},
$$
and using the boundedness of $v'$ again, we obtain a uniform lower bound for  $v$. Thus, an open bounded set $\Omega$ is defined as before, now with $r=-R$.

Next, observe that the map $\Phi$ vanishes only in one point $(u,v)\in \R^2$, namely when 
$$e^u=K,\qquad e^v= \frac{f(K)}{\frac \varphi{\nu+K} - \sigma}.
$$
Such a point belongs to $\Omega$ and, since $\Phi$ is smooth, we may compute 
$$D\Phi(u,v)=
\left(
\begin{array}{cc}
     M & \left(\sigma - \frac\varphi{\nu + K}\right)e^v \\
     \frac{\gamma\nu K}{(\nu+K)^2} & 0
\end{array}
\right),
$$
where the quantity $M=M(u,v)$ is irrelevant for our purposes since, 
according to the properties of the Brouwer degree, 
we only need to verify that the previous matrix is invertible. Indeed, 
from the hypotheses, it is seen that $\det(D\Phi(u,v))>0$, so we deduce that 
$$\deg(\Phi, (-R,R)\times (-M,M),0) =1,$$
and the continuation theorem applies. 
\end{proof}


\section{Interpretations of the mathematical results}\label{sec: strategy control crime}
The analytical results obtained in Sections \ref{sec:linear stab auton} and \ref{sec: non auton periodic sols} provide a mathematical basis for understanding the mechanisms that may influence the long-term dynamics of criminal activity and criminal populations. In this section, we interpret these results from the perspective of crime prevention and control, with the aim of illustrating how the insights derived from the mathematical analysis may inform the design and implementation of public policies. While the model does not seek to prescribe specific policy interventions, its results may provide a useful theoretical basis for policymakers and experts in criminology, law enforcement, and public policy to identify relevant mechanisms, assess potential intervention strategies, and support evidence-based approaches to crime prevention and reduction.

In the autonomous case, Proposition \ref{prop auton criminal-free stability} provides a threshold condition for the 
local stability of the criminal-free equilibrium $(N^\dagger,0)$. More 
precisely, this equilibrium is locally asymptotically stable when
\[
    \frac{\gamma N^\dagger}{\nu+N^\dagger}<\eta+l_e.
\]
Thus, the combined effect of natural removal and law enforcement, represented 
by $\eta+l_e$, must exceed the effective criminalization rate evaluated at 
the criminal-free population level $N^\dagger$. Conversely, when
\[
    \frac{\gamma N^\dagger}{\nu+N^\dagger}>\eta+l_e,
\]
the criminal-free equilibrium becomes unstable. Hence, within the model, the 
stability of the criminal-free state is determined by the balance between 
criminalization and the mechanisms contributing to the reduction of the 
criminal population. The delay $\tau$ plays a different role in the dynamics of the coexistence 
equilibrium. As shown in Proposition \ref{prop: estab auton coexist}, changes in $\tau$ may modify its local 
stability, and a critical delay may separate stable and unstable regimes. 
Since $\tau$ represents the latency between exposure to criminal behavior and 
the subsequent behavioral response, this result indicates that the timing of 
such responses can influence the persistence and stability of coexistence 
between criminal and non-criminal populations. In particular, sufficiently 
large delays may destabilize an otherwise stable coexistence state and give 
rise to oscillatory behavior in the corresponding linearized dynamics.
We next consider the non-autonomous case analyzed in Section \ref{sec: non auton periodic sols}. Theorems \ref{teo sol per 1}
and \ref{teo sol per 2} provide sufficient conditions for the existence of strictly positive 
$T$-periodic solutions under strong and weak counteroffensive measures, 
respectively. A condition common to both results is
\[
    \gamma>\eta+\overline{l_e},
\]
where $\overline{l_e}$ denotes the average law-enforcement intensity over one 
period. Thus, within the hypotheses of these existence theorems, the 
criminalization rate must exceed the combined effect of natural removal and 
average law enforcement. The resulting positive periodic solutions represent 
recurrent coexistence regimes in which both the criminal and non-criminal 
populations persist and fluctuate over time.
It is worth noting that the stability threshold for the criminal-free 
equilibrium in the autonomous case does not explicitly involve the parameter 
$\sigma$. Similarly, the condition $\gamma>\eta+\overline{l_e}$, which 
appears in both periodic-existence results, has the same form under the strong 
and weak counteroffensive regimes considered in Section \ref{sec: non auton periodic sols}. This does not imply 
that the community-based response represented by $\sigma$ is dynamically 
irrelevant; rather, its role is reflected in the additional hypotheses under 
which the corresponding existence results are obtained.
Overall, the analysis distinguishes three mechanisms with different 
dynamical effects: the balance between criminalization and removal determines 
the local stability of the criminal-free equilibrium, the behavioral delay 
may alter the stability of coexistence, and periodic law enforcement is 
compatible, under the conditions established in Section \ref{sec: non auton periodic sols}, with recurrent 
positive dynamics. These interpretations are qualitative and should be 
understood within the assumptions and scope of the proposed model.

\section{Numerical simulations} \label{sec:numerical sims}

This section presents numerical simulations for both the autonomous and non-autonomous cases. The aim is to illustrate the theoretical results obtained in Sections  \ref{sec:linear stab auton} and \ref{sec: non auton periodic sols} and explore some dynamical features of the model numerically. The simulations are performed using the \texttt{DifferentialEquations.jl} (see \cite{rackauckas2017differentialequations}) and \texttt{Plots.jl} (see \cite{tom_breloff_2025_15047900}) libraries in the \texttt{Julia} programming language (see \cite{Julia-2017}). 

\subsection{Numerical simulations in the autonomous case.} We first consider the autonomous system \eqref{autonomous main eq} and illustrate the stability properties of the criminal-free and coexistence equilibria.

If we look at the criminal-free equilibrium $(N_{\dagger},0)$, where $N_{\dagger}$ is the unique positive zero of $f(N)$ (see system \eqref{main eq 1}, for example $N_{\dagger}=m$ if $f(N)=m-N$, that is, in the logistic case), the stability depends on $\dfrac{\gamma N_{\dagger}}{\nu+N_{\dagger}} \gtrless \eta+l_e$ as stated in Proposition \ref{prop auton criminal-free stability}. If we define $\mathcal{R}=N_{\dagger}g_{\gamma}(N_{\dagger})/(\eta+l_e)$, we are interested in whether $\mathcal{R} \gtrless 1$. Since stability changes do not depend on the delay, we can perform a delay-independent stability analysis in terms of key criminalization parameters. Figure \ref{fig auton eq criminal-free gamma and le sensitivity} illustrates this analysis for parameters $\gamma$ and $l_e$, as the system crosses the critical threshold $\mathcal{R}=1$, for the stability of the criminal-free equilibrium ($\mathcal{R}<1$ stable, $\mathcal{R}>1$ unstable). The figure displays the last $1000$ values of $N$ and $C$, obtained after allowing the system to evolve over a sufficiently long time interval.
\begin{figure}[ht]
\centering
\begin{subfigure}{0.49\textwidth}
    \centering
    \includegraphics[width=\textwidth]{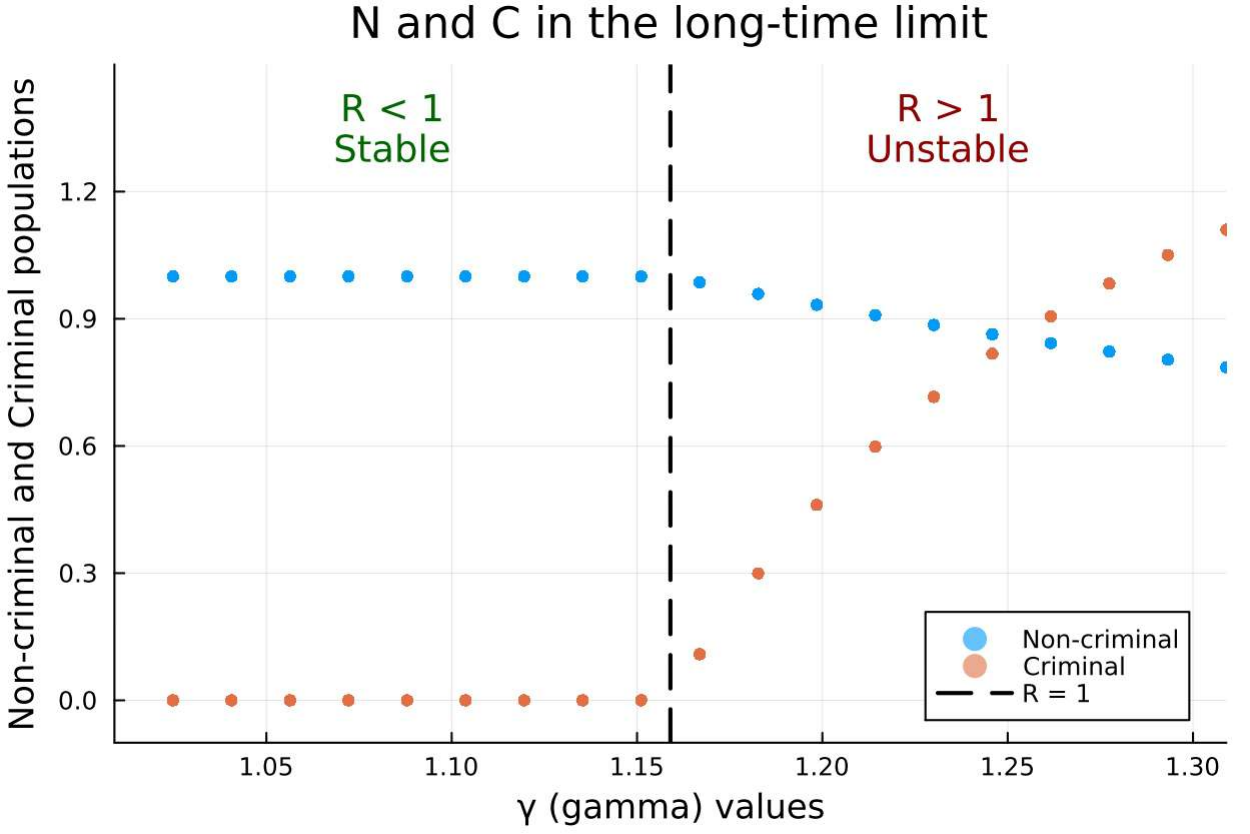}
    \caption{Parameter $\gamma$ varies, $l_e = 0.51$.}
\end{subfigure}
    \hfill
\begin{subfigure}{0.49\textwidth}
    \centering
    \includegraphics[width=\textwidth]{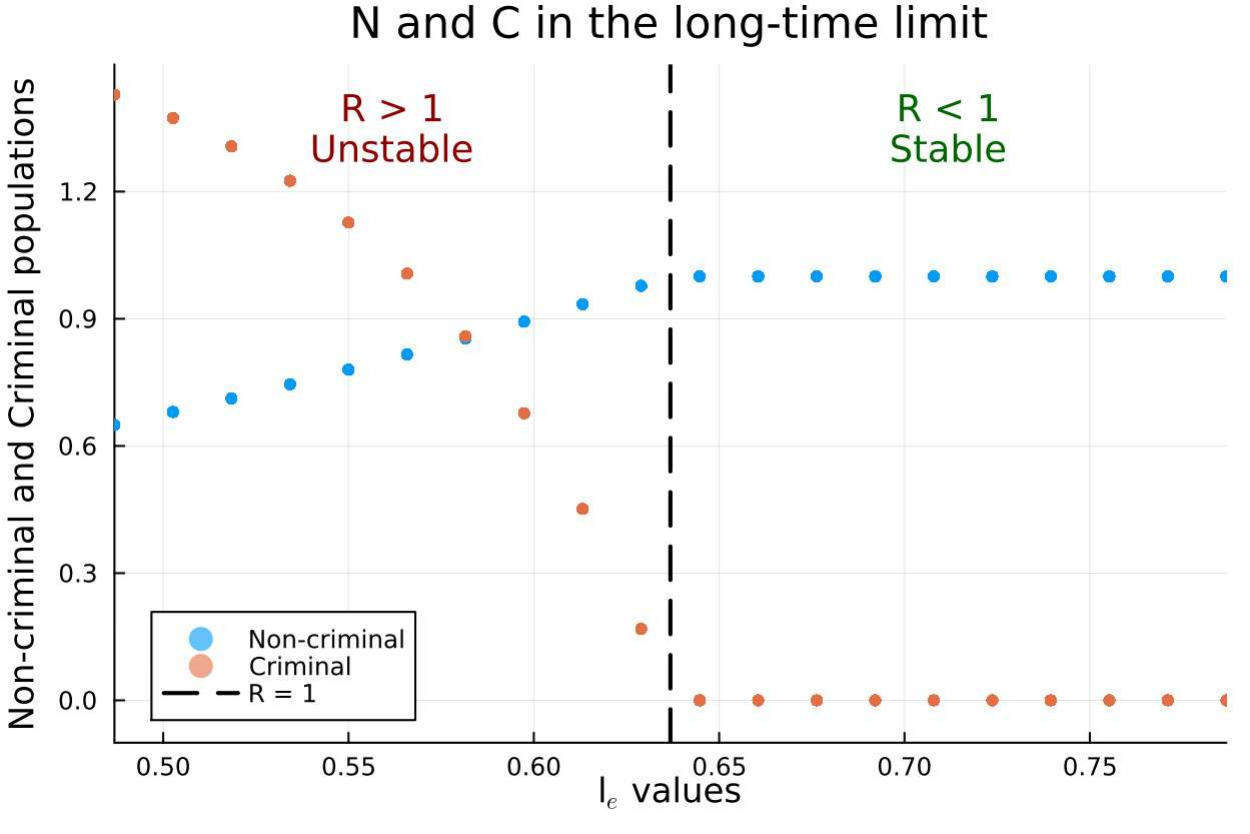}
    \caption{Parameter $l_e$ varies, $\gamma = 1.4$.}
\end{subfigure}
\caption{Simulation of the autonomous system \eqref{autonomous main eq}  when $\tau=2$, with $f(N)=1-N$, $\varphi=1$, $\nu=0.9$, $\sigma=0.4$, $\eta=0.1$ and initial condition $(N_0(t), C_0(t))=(0.5, 0.5), -\tau\leq t\leq 0$. We observe that the criminal-free equilibrium $(N_{\dagger}=1, C=0)$ changes stability as critical threshold $\mathcal{R}=1$ is crossed: $\mathcal{R}<1$ corresponds to stability and $\mathcal{R}>1$ corresponds to instability (See Proposition \ref{prop auton criminal-free stability}).} 
\label{fig auton eq criminal-free gamma and le sensitivity}
\end{figure}

We next consider the coexistence equilibrium $(\hat{N},\hat{C})$ (see Section \ref{sec:linear stab auton}), here we have stability under the conditions of Proposition \ref{prop: estab auton coexist}. Figure \ref{fig auton tau critico eqCoex} illustrates the stability change as the delay crosses its critical threshold.  The figure displays the last $1000$ values of $N$ and $C$, obtained after allowing the system to evolve over a sufficiently long time interval.
\begin{figure}[ht]
\centering
\includegraphics[width=0.7\textwidth]{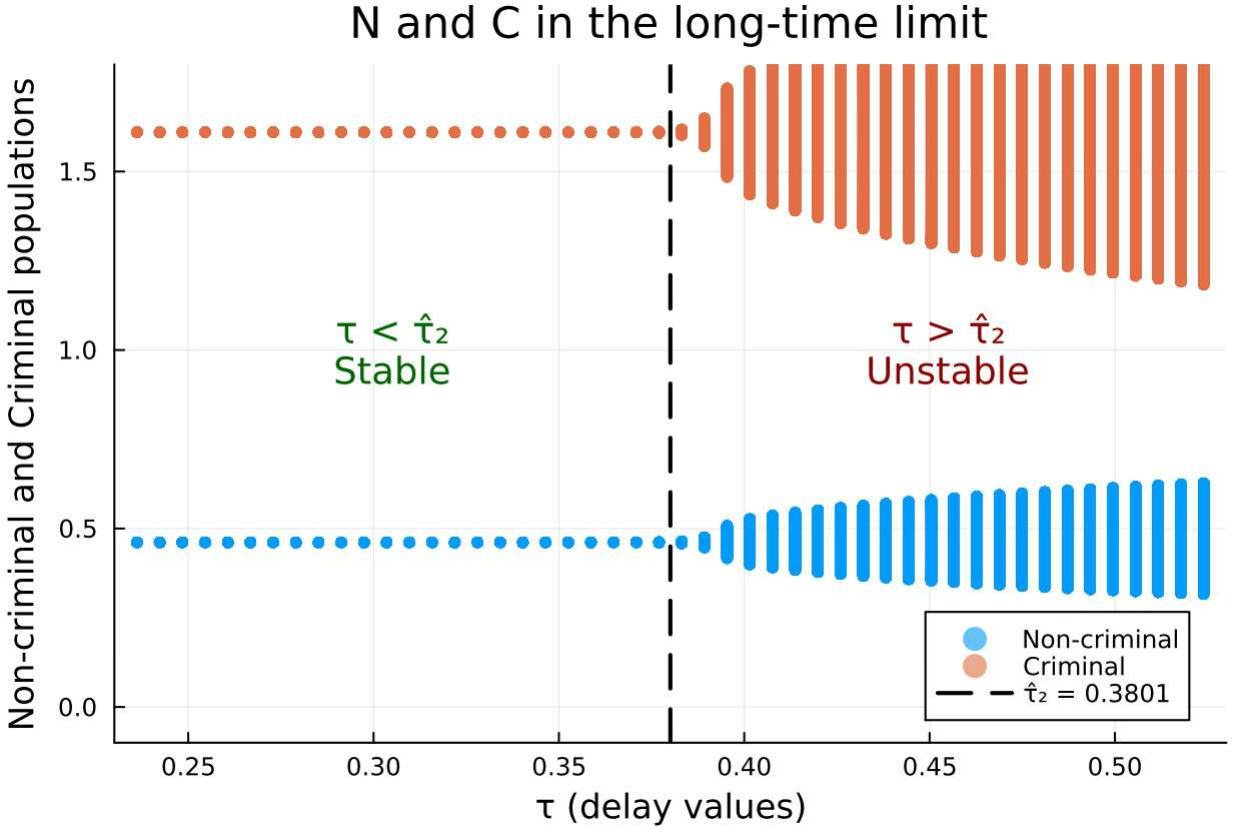}
\caption{Simulation of the autonomous system \eqref{autonomous main eq} with initial conditions $(N_0(t), C_0(t))$, $-\tau\leq t\leq 0$, close to the coexistence equilibrium $(\hat{N},\hat{C})\simeq (0.4613, 1.61)$, with the same parameter values as in Figure \ref{fig auton eq criminal-free gamma and le sensitivity}, except for $\gamma=1.8$. We observe a change in the stability of the coexistence equilibrium $(\hat{N},\hat{C})$ as the delay crosses the critical threshold (see Proposition \ref{prop: estab auton coexist}).}
\label{fig auton tau critico eqCoex}
\end{figure}

Finally, for the non-delayed case, we show in Figure \ref{fig auton non-delay direction field} the direction field for the asymptotically stable criminal-free equilibrium (with the same parameter values as Figure \ref{fig auton eq criminal-free gamma and le sensitivity}) and for the asymptotically stable coexistence equilibrium (same parameter values as Figure \ref{fig auton tau critico eqCoex}).
\begin{figure}[ht]
\centering
\begin{subfigure}{0.49\textwidth}
    \centering
    \includegraphics[width=\textwidth]{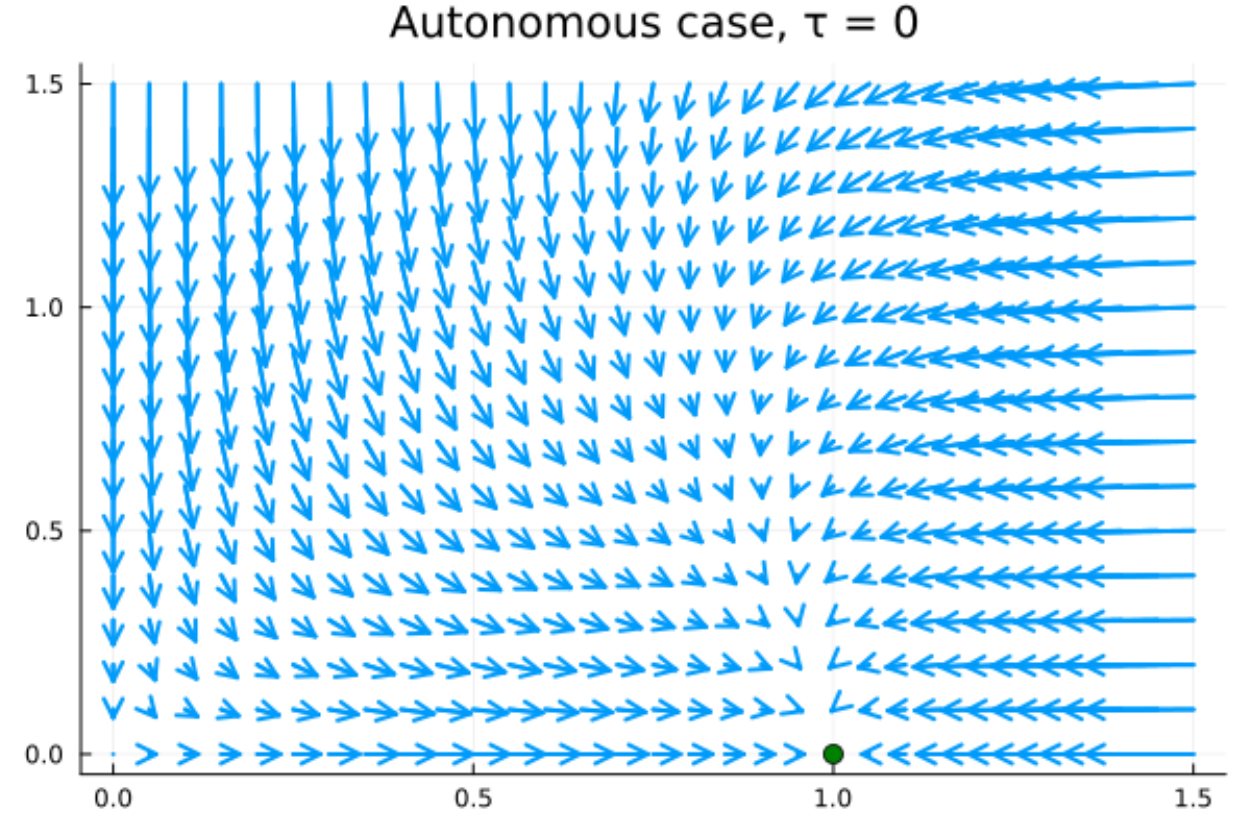}
    \caption{Direction field for $\tau=0$ and asymptotically stable criminal-free equilibrium (parameters same as Fig. \ref{fig auton eq criminal-free gamma and le sensitivity}.)} 
\end{subfigure}
    \hfill
\begin{subfigure}{0.49\textwidth}
    \centering
    \includegraphics[width=\textwidth]{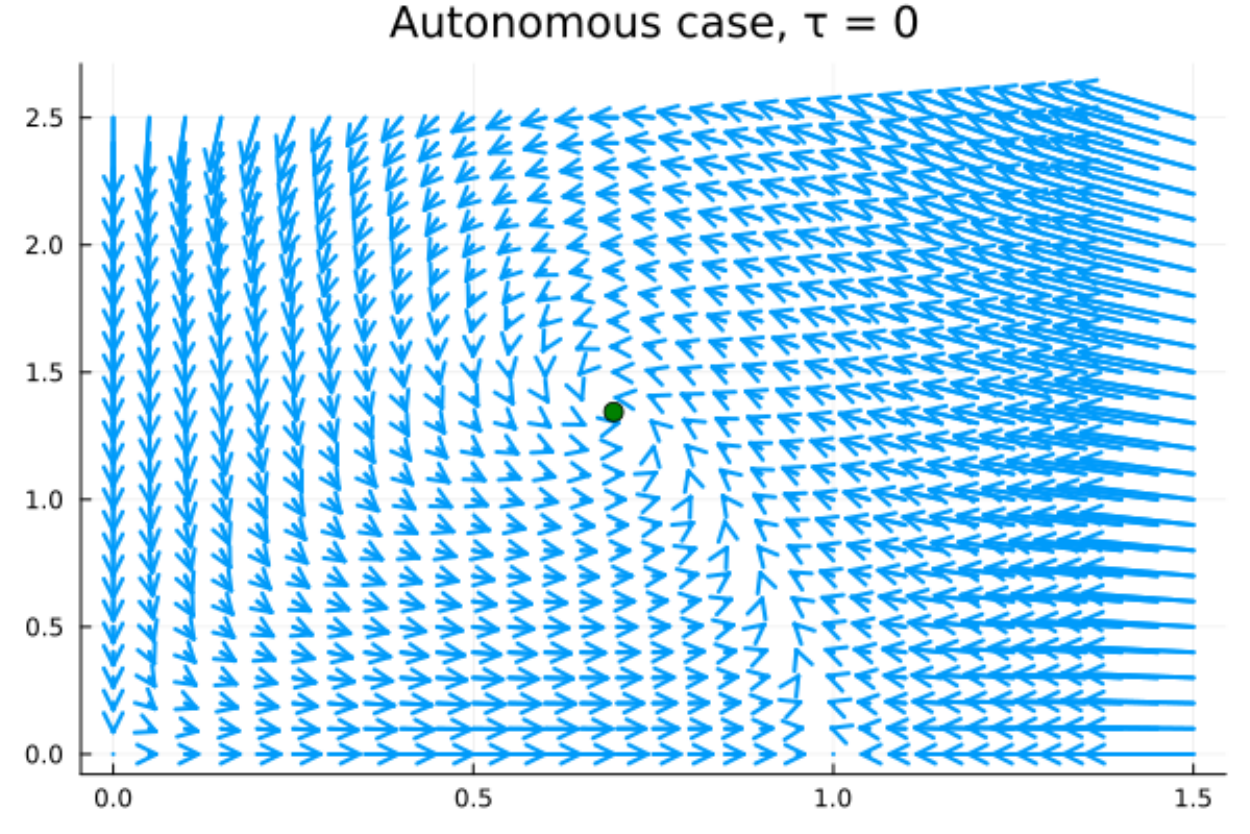}
    \caption{Direction field for $\tau=0$ and asymptotically stable coexistence equilibrium (parameters same as Fig. \ref{fig auton tau critico eqCoex}.)}
\end{subfigure}
\caption{Direction fields of the autonomous system \eqref{autonomous main eq}  for the non-delayed case.}
\label{fig auton non-delay direction field}
\end{figure}

The trivial equilibrium is not considered numerically, since Section 3.1 shows that it is unstable for every $\tau \geq 0$. Moreover, it has limited relevance to the population dynamics considered here, as both the criminal and non-criminal populations vanish at this state.


\subsection{Comparison with a mass-action interaction term}

To illustrate the dynamical effect of the saturation mechanism incorporated into system \eqref{main eq 1}, we compare the Holling type II functional response with a standard mass-action formulation. For this purpose, consider the alternative autonomous system
\[
\begin{split}
\dot{N}(t)&=
N(t)f(N(t))-\beta N(t)C(t)+\sigma N(t)C(t),\\
\dot{C}(t)&=
-\eta C(t)-l_eC(t)
+\alpha N(t-\tau)C(t-\tau),
\end{split} \qquad (M_{A})
\]
where the Holling type II interaction terms in \eqref{main eq 1} are replaced by their bilinear mass-action counterparts. To obtain a meaningful comparison, the coefficients $\beta$ and $\alpha$ are chosen so that the corresponding interaction rates coincide at the criminal-free population level $N=N_{\dagger}$, where $f(N_{\dagger})=0$. Thus, $g_{\varphi}(N_{\dagger})=\beta N_{\dagger}$ and $g_{\gamma}(N_\dagger)=\alpha N_{\dagger}$ imply
\[
\beta=\frac{\varphi}{\nu+N_{\dagger}},
\quad \text{and}\quad
\alpha=\frac{\gamma}{\nu+N_{\dagger}},
\]
respectively. For these values, both formulations have the same interaction rates at $N=N_{\dagger}$. However, the two formulations differ as the non-criminal population increases. In the mass-action formulation, the interaction rates per criminal individual increase linearly with $N$, whereas the Holling type II formulation incorporates saturation and keeps these rates bounded.
\noindent
For the numerical comparison, we take as before
\[
f(N)=1-N, \quad \varphi=1, \quad \nu= 0.9, \quad \sigma=0.4 \quad \text{and} \quad \gamma=1.4,
\]
together with 
\[
\eta=0.1, \quad l_e=0.51 \quad \text{and} \quad \tau=2.
\]
For these values, the calibrated mass-action coefficients are given by
\[
\beta=10/19 \quad \text{and} \quad \alpha=14/19.
\]
For the simulation, both systems are solved using the same constant initial condition
\[
(N_0(t),C_0(t))=(3,1), \quad \forall t\in [-\tau,0].
\]
Since $N_0>1=N_{\dagger}$, the initial dynamics take place in a regime where the difference between the linear and saturating interaction dynamics is already significant.
\begin{figure}[t]
\centering
\includegraphics[width=0.9\textwidth]{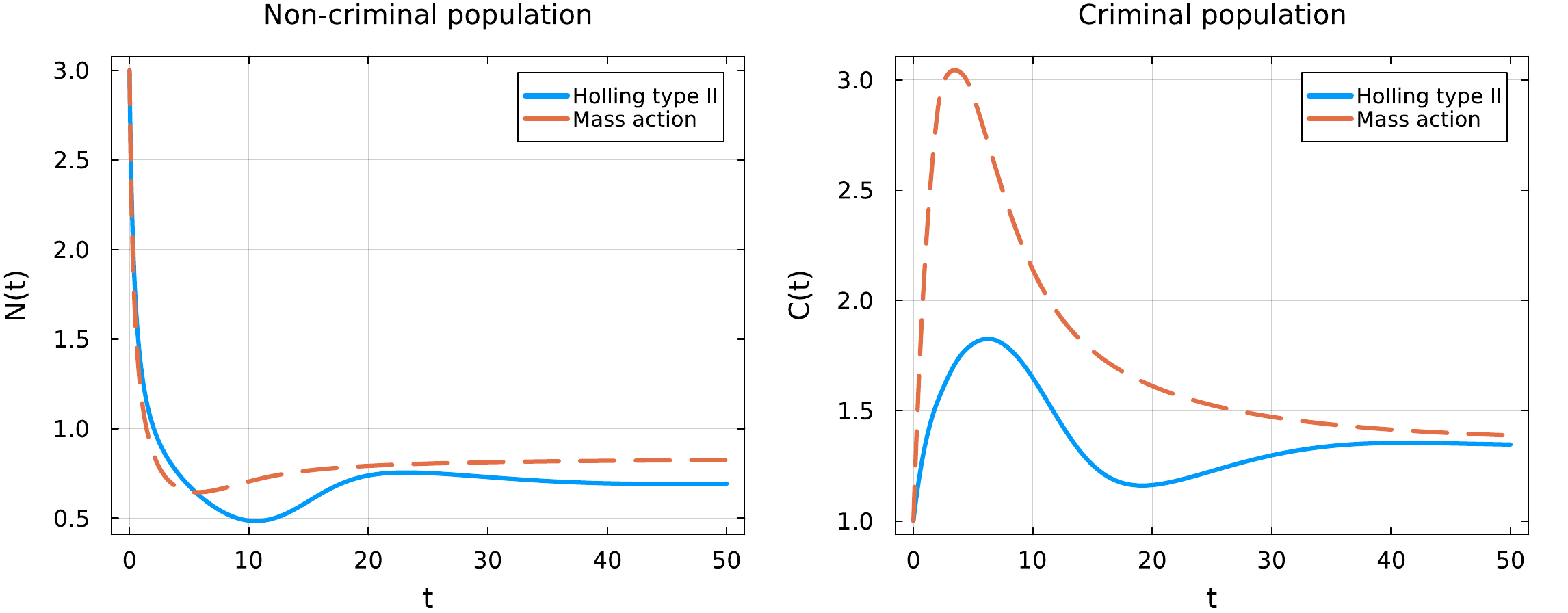}
\caption{Comparison between the Holling type II and mass-action formulations for $f(N)=1-N, \, \varphi=1, \, \nu= 0.9,\, \sigma=0.4, \,\gamma=1.4,\,\eta=0.1, \, l_e=0.51, \tau=2,$ with constant initial history $(N_0(t),C_0(t))=(3,1)$ for all $t\in [-\tau,0]$. Solid and dashed curves correspond to the Holling type II and mass-action formulations, respectively.}
\label{fig comparison holling vs mass-action}
\end{figure}

\noindent
Figure \ref{fig comparison holling vs mass-action} compares the resulting trajectories. The effect of the interaction mechanism is particularly evident in the criminal population. Although the two interaction functions are calibrated to coincide at $N=N_{\dagger}$, for $N>N_{\dagger}$ the mass-action criminalization rate satisfies
\[
\alpha N>\frac{\gamma N}{\nu+N}.
\]
Consequently, the mass-action model produces a substantially larger transient increase in $C(t)$ whereas the Holling type II formulation limits the magnitude of this response through saturation. This difference is reflected in the numerical trajectories, with the mass-action formulation producing a markedly higher transient peak in the criminal population.
\noindent
Overall, this comparison highlights the modeling role of the Holling type II functional response. By imposing an upper bound on the interaction rate per criminal individual, the saturation mechanism prevents interaction intensity from increasing indefinitely with the size of the non-criminal population. The numerical results illustrate how this structural difference can substantially affect the predicted dynamics, particularly during transient regimes in which $N$ lies above the reference level $N_{\dagger}.$

\subsection{Numerical simulations in the non-autonomous case.} 

We present numerical simulations of the non-autonomous system~(1) for the 
non-delayed case $\tau=0$ and the delayed case $\tau=2$; see Figure \ref{fig non-auton eq Per}.
In both cases, the numerical trajectories approach a periodic regime. 
However, the delay modifies both the transient and asymptotic dynamics: 
the limiting periodic responses differ, convergence to the asymptotic regime 
is slower in the delayed case, and additional harmonic components appear in 
the numerical solution. The latter effect also depends on the temporal 
structure of the periodic law-enforcement function $l_e(t)$.

We present two simulations of system \eqref{main eq 1}: non-delay and delay cases in Figure \ref{fig non-auton eq Per}. In both situations, we observe the asymptotic behavior towards a periodic solution. It is interesting to note some differences between the cases with and without delay: the limiting periodic solution is not the same; the delay makes the asymptotic behavior occur in a longer time; and finally that the delay also has the impact of making more harmonics appear in the limit solution (note that in the case without delay the limiting $N(t)$ and $C(t)$ are simple harmonics), although this also depends on the $l_e(t)$ periodic function of law enforcement.

\subsection{Comparison with a mass-action interaction term in the non-autonomous case}
We now extend the comparison to the non-autonomous setting by considering the periodic law-enforcement function $l_e(t)=0.2 \sin(t/4)+0.5$ whose fundamental period is $T=8\pi$. All remaining parameter values are kept unchanged from the previous comparison, and the same calibration $\beta=10/19$ and $\alpha=14/19$ is used for the mass-action formulation. Now, both systems are simulated using the constant initial history
\[
(N_0(t),C_0(t))=(1.5,1.5), \quad \forall t\in [-\tau,0].
\]
with $\tau=2.$ Figure \ref{fig comparison holling vs mass-action non-autonomous} shows the long-term dynamics of both formulations over the interval $t\in[300,500]$, after the initial transient and illustrates their asymptotic behavior. In contrast with the autonomous case, the periodic forcing generates persistent oscillations in both formulations. However, the resulting asymptotic responses exhibit markedly different structures. The mass-action model approaches a periodic regime synchronized with the external forcing, with period $T=8\pi$. The Holling type II model, on the other hand, numerically approaches a subharmonic response of period $2T=16\pi$. In particular, consecutive oscillations alternate between two distinct profiles, so that the solution approximately repeats itself only after two periods of $l_{e}(t)$. To verify this behavior numerically, let
\[
X_{H}(t)=(N_{H}(t),C_{H}(t)) \quad \text{and} \quad X_{M_A}(t)=(N_{M_A}(t),C_{M_A}(t)),
\] 
denote the solutions corresponding to the Holling type II formulation \eqref{main eq 1} and mass-action formulation $(M_A)$, respectively. Over the late-time interval $I=[350,449]$ we compute
\[
\Delta_{H,s}:=\max_{t\in I}\|X_{H}(t+sT)-X_{H}(t)\| \quad \text{and} \quad \Delta_{M_A,s}:=\max_{t \in I}\|X_{M_A}(t+sT)-X_{M_A}(t)\|,
\]
for $s=1,2.$ The numerical results are:
\[
\Delta_{H,1}\approx 9.07 \times 10^{-1} \quad \text{and} \quad \Delta_{H,2}\approx 1.87 \times 10^{-5}.
\]
Thus, the solution $X_{H}(t)$ does not repeat after one forcing period, whereas the discrepancy after two periods is several orders of magnitude smaller.
For the mass-action model, we obtain
\[
\Delta_{M_A,1}\approx 3.03 \times 10^{-6} \quad \text{and} \quad \Delta_{M_A,2}\approx 2.61 \times 10^{-6},
\]
showing that the solution already repeats, up to numerical accuracy, after a single forcing period. The mass-action response is therefore numerically consistent with a $T$-periodic asymptotic regime. These simulations illustrate that the choice of the interaction mechanism can influence not only the magnitude of the population response, but also the structure of the resulting long-term dynamics under periodic forcing. For the parameter values considered here, the Holling type II formulation leads to a subharmonic response with period $2T$, whereas the mass-action model approaches a $T$-periodic regime synchronized with the external forcing. Hence, replacing the saturating interaction by its bilinear counterpart may qualitatively alter the asymptotic response of the system, even when both formulations are calibrated to coincide at the same reference population level.

\begin{figure}[ht]
\centering
\begin{subfigure}{0.43\textwidth}
    \centering
    \includegraphics[width=\textwidth]{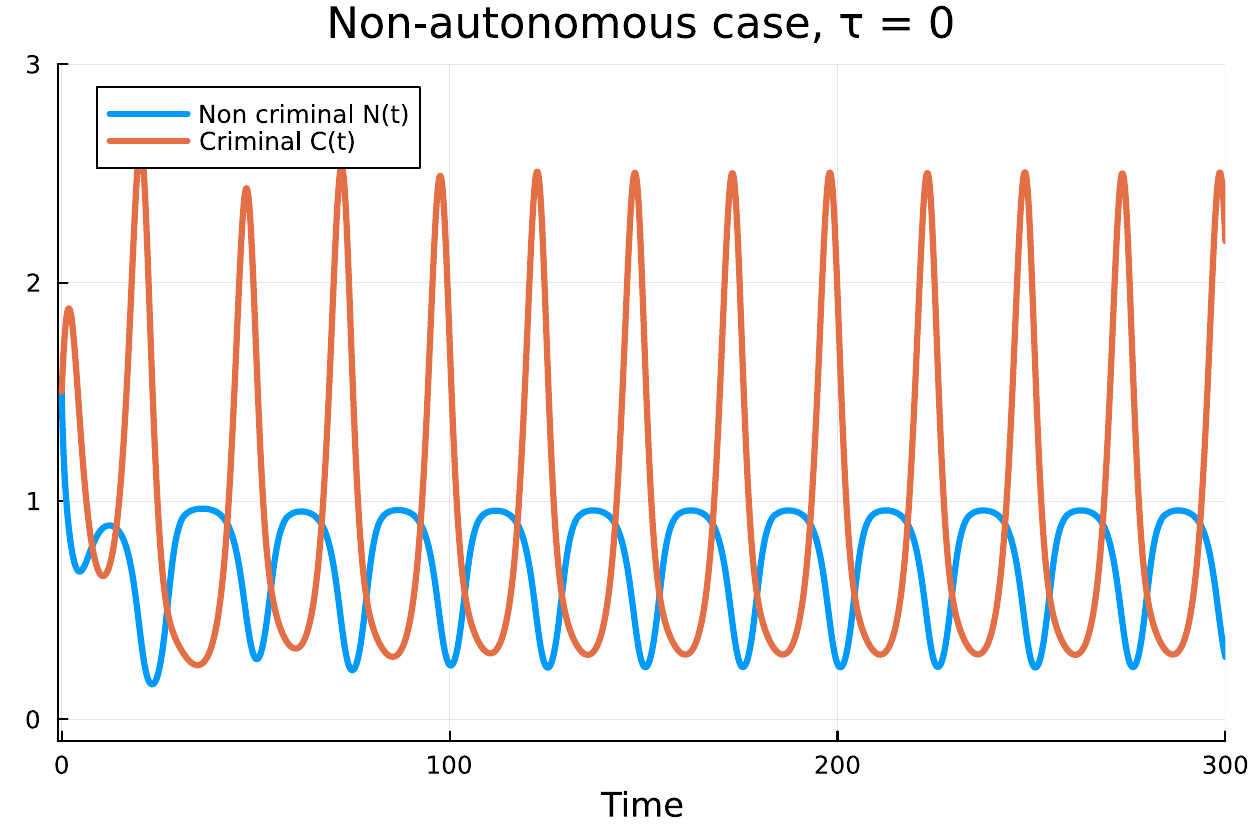}
    \caption{time vs $N$ and $C$, $\tau=0$.}
\end{subfigure}
    \hfill
\begin{subfigure}{0.43\textwidth}
    \centering
    \includegraphics[width=\textwidth]{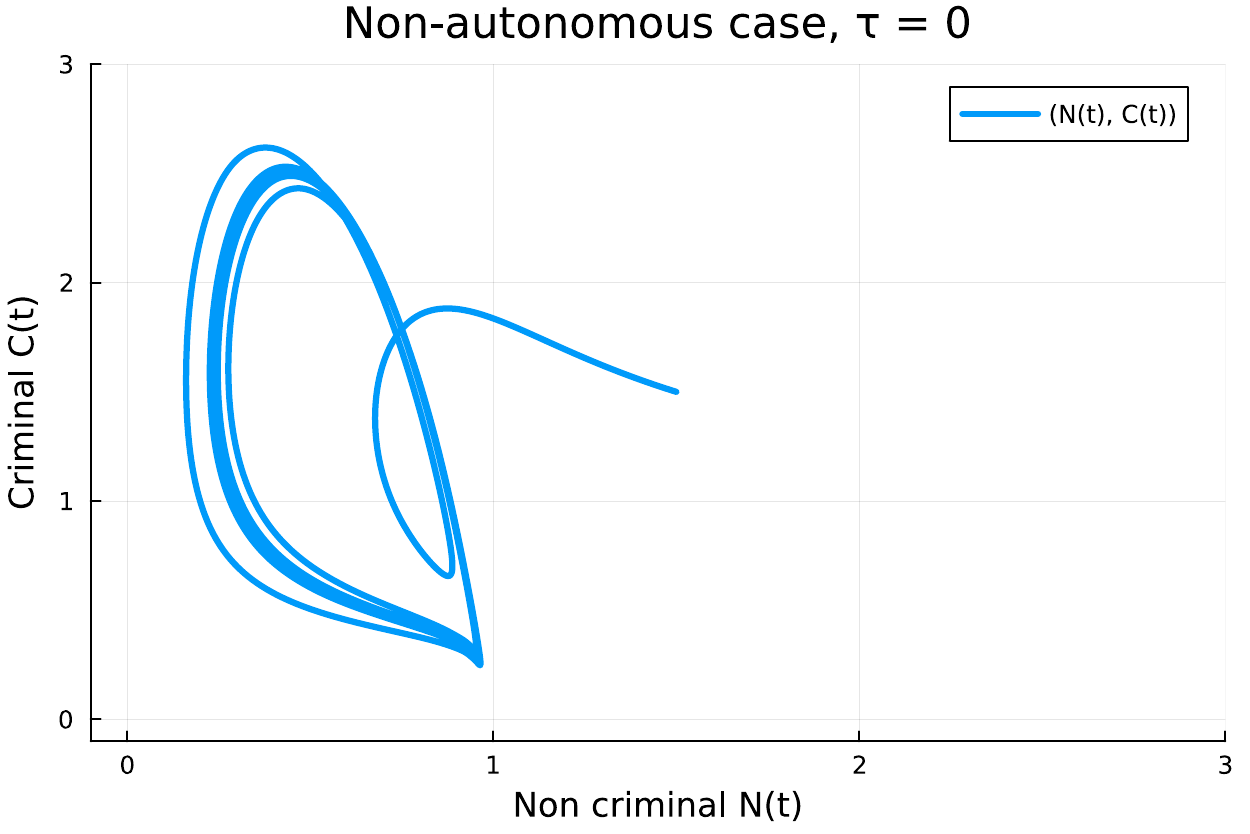}
    \caption{$N$ vs $C$, $\tau=0$.}\label{}
\end{subfigure}
    \hfill\\[2ex]
\begin{subfigure}{0.43\textwidth}
    \centering
    \includegraphics[width=\textwidth]{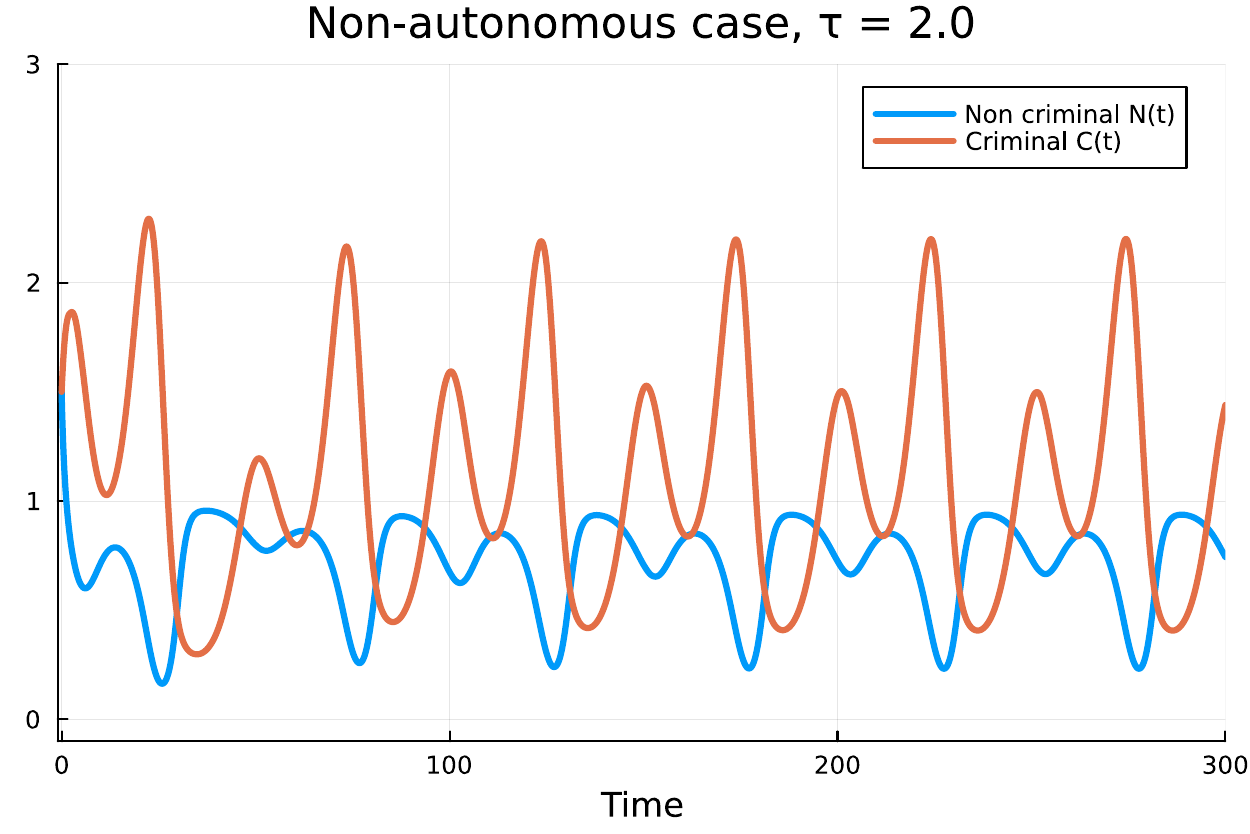}
    \caption{time vs $N$ and $C$, $\tau=2$.}\label{}
\end{subfigure}
    \hfill
\begin{subfigure}{0.45\textwidth}
    \centering
    \includegraphics[width=\textwidth]{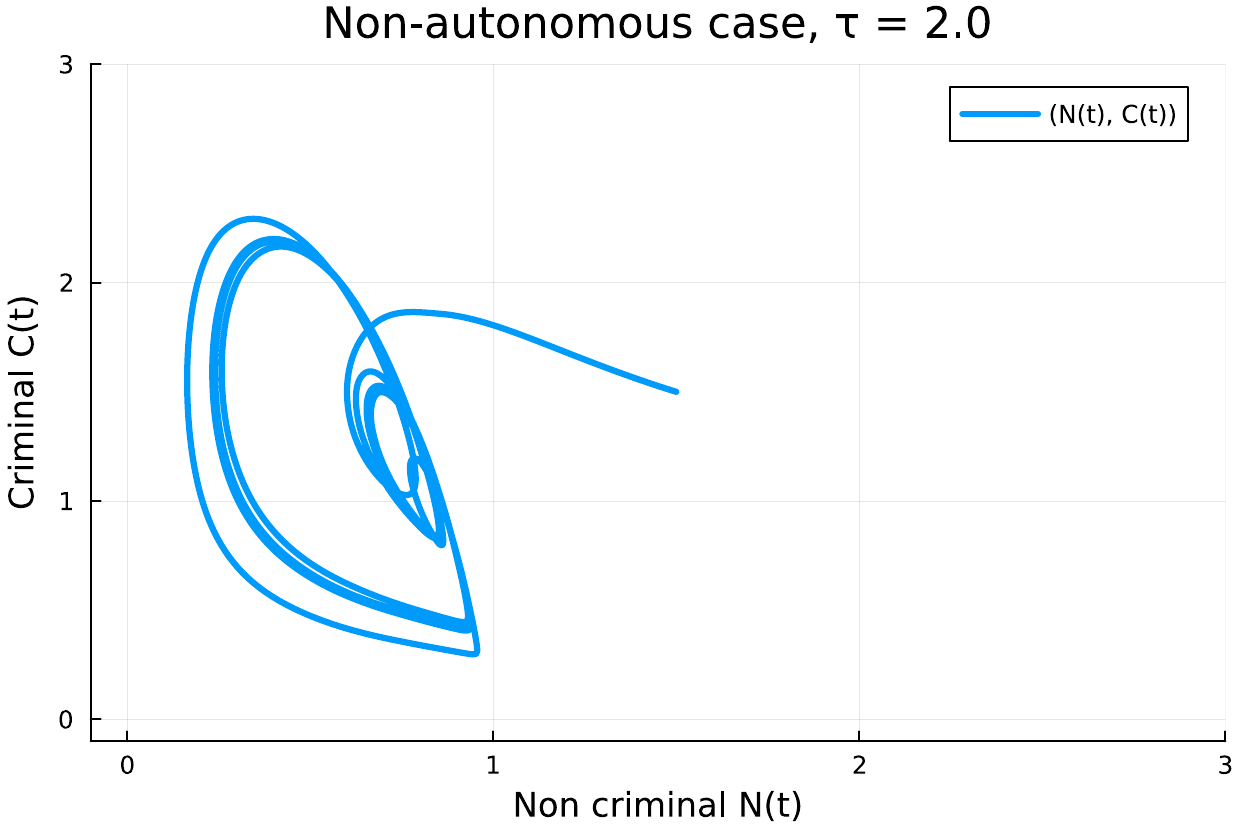}
    \caption{$N$ vs $C$, $\tau=2$.}\label{}
\end{subfigure}
\caption{Simulation of the non-autonomous system \eqref{main eq 1}  when $\tau=0$ and $\tau=2$, with $l_{e}(t)=0.2\sin(t/4)+0.5$ and initial condition $(N_0(t), C_0(t))=(1.5, 1.5), -\tau\leq t\leq 0$ (with the remaining parameter values as in Figure \ref{fig auton eq criminal-free gamma and le sensitivity}). We observe the asymptotic behavior toward a periodic solution.}
\label{fig non-auton eq Per}
\end{figure}

\begin{figure}[ht]
\centering
\includegraphics[width=0.9\textwidth]{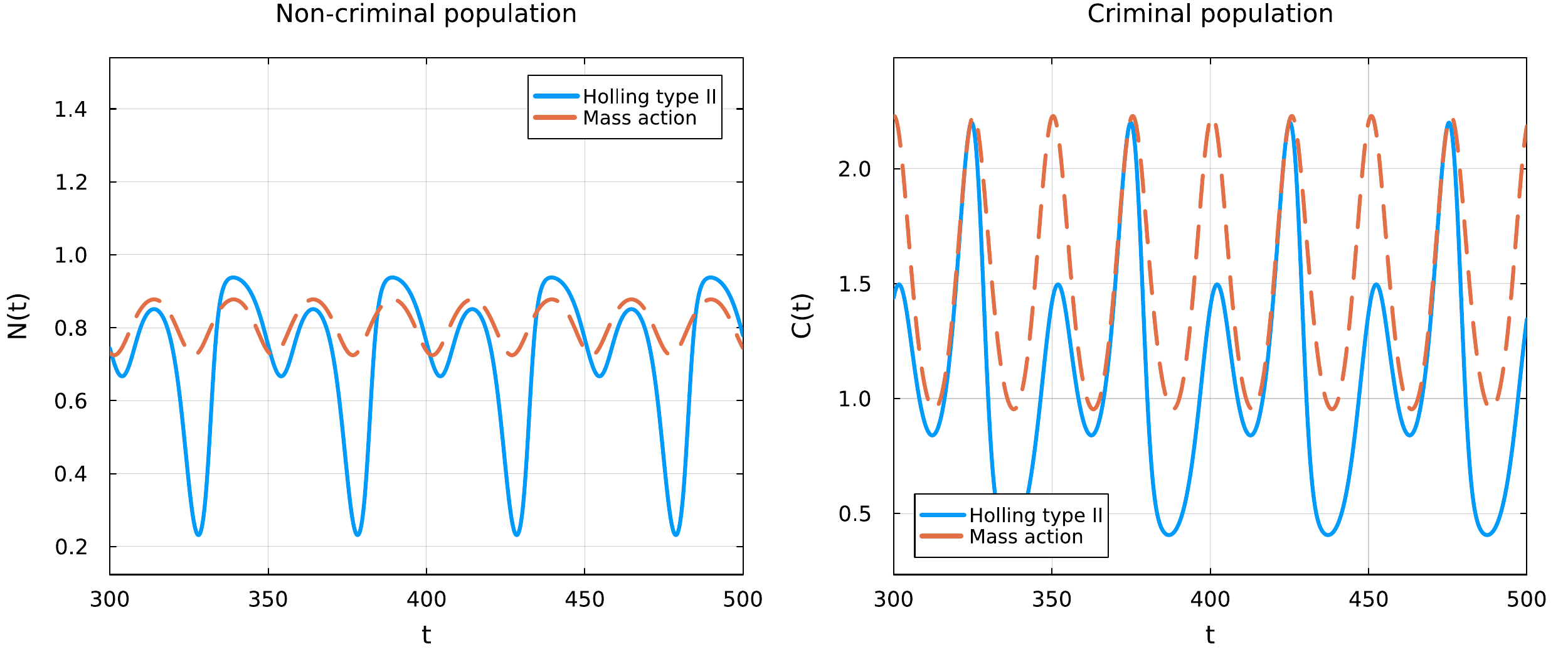}
\caption{Long-term comparison between the Holling type II and mass-action formulations under periodic law enforcement $l_e(t)=0.2\sin(t/4)+0.5$, with forcing period $T=8\pi$, $\tau=2$, and initial condition $(N_0(t), C_0(t))=(1.5, 1.5), -\tau\leq t\leq 0$.  The trajectories are shown over $t\in[300,500]$, after the initial transient. The mass-action formulation $(M_A)$ approaches a $T$-periodic response, whereas the Holling type II formulation \eqref{main eq 1} exhibits a numerically observed subharmonic response of period $2T$.}
\label{fig comparison holling vs mass-action non-autonomous}
\end{figure}

\section*{Concluding Remarks and Future Perspectives}
A nonlinear time-delay model has been formulated and analyzed to describe the interactions between criminal and non-criminal populations, incorporating social influence, saturation effects represented by a Holling type II functional response, and time-dependent law enforcement. The inclusion of a discrete time delay accounts for the latency between social exposure and subsequent behavioral change, thereby introducing memory effects as an intrinsic feature of the system.

For the autonomous model, explicit threshold conditions have been established that characterize the persistence or eradication of the criminal population. These conditions identify the balance between recruitment mechanisms and law-enforcement effects that determines the stability of the criminal-free equilibrium. In addition, the analysis demonstrates that the time delay may significantly alter the stability properties of the system, leading to stability switches and oscillatory dynamics. Thus, the delay constitutes not merely a modeling refinement, but a relevant mechanism capable of generating qualitative changes in the long-term behavior of the system.

For the non-autonomous model with periodic law enforcement, the analysis establishes the existence of strictly positive periodic solutions under appropriate conditions involving the average effectiveness of the enforcement mechanism. This result demonstrates that periodic variations in law enforcement may coexist with persistent recurrent dynamics of the criminal and non-criminal populations. In particular, the existence of such periodic regimes is governed by conditions involving the average enforcement effectiveness, while the temporal structure of the forcing contributes to the resulting dynamical behavior. These findings provide a mathematical basis for assessing the role of sustained and periodically varying intervention strategies within the proposed framework.

Despite these results, several limitations should be acknowledged. The model assumes homogeneous mixing between the populations, deterministic interactions, and a single discrete time delay. These assumptions may not fully capture the heterogeneity of individual behavioral responses, spatial effects, or the multiple temporal scales that may characterize the evolution of criminal behavior. Several extensions are therefore possible, including the incorporation of distributed or state-dependent delays, spatially structured or network-based interactions, stochastic effects, and control-theoretic approaches for the analysis and optimization of law-enforcement strategies.

Overall, the results emphasize the fundamental role of memory and temporal variability in determining the qualitative dynamics of nonlinear social systems. The analysis provides a unified framework in which threshold behavior, delay-induced stability changes, oscillatory dynamics, and periodic forcing can be studied within a common mathematical model. These findings contribute to the broader analysis of nonlinear delay differential equations with time-dependent coefficients and provide a mathematical foundation for further investigations of recurrent dynamics in socially driven population models.

\appendix{}
\section{Characteristic equation and (local) stability}
Consider the following linear system of delay differential equations
\begin{equation}\label{appendix linear system with one delay}
\dot{Z}(t) = AZ(t) + BZ(t-\tau),
\end{equation}
with  $A,B \in \mathbb{M}_{n\times n}$ constant real matrices, $Z=Z(t)\in \mathbb{R}^{n}$ and $\tau \geq 0$. We seek non-trivial exponential solutions of \eqref{appendix linear system with one delay}, say $Z(t)=\text{e}^{\lambda t}v$ with $\lambda\in\C$ and $v\in\C^n$, $v\neq 0$. Then equation \eqref{appendix linear system with one delay} becomes
\begin{equation*}\label{appendix linear system charact eq deduction}
\lambda \text{e}^{\lambda t}v = A\text{e}^{\lambda t}v + B\text{e}^{\lambda (t-\tau)}v \quad \Leftrightarrow \quad 
\lambda v = Av + B\text{e}^{-\lambda \tau}v,
\end{equation*}
which is equivalent to
\begin{equation}\label{appendix linear system charact eq deduction}
\big(\lambda I - A - \text{e}^{-\lambda \tau}B\big)v = 0.
\end{equation}
Equation \eqref{appendix linear system charact eq deduction} is equivalent to
\begin{equation}\label{appendix linear system charact eq}
\begin{split}
\det\big(\lambda I - A - \text{e}^{-\lambda \tau}B\big) &= 0.
\end{split}
\end{equation}
We refer to equation \eqref{appendix linear system charact eq} as the \textit{characteristic equation} of the delay differential equation \eqref{appendix linear system with one delay}, and solutions of equation \eqref{appendix linear system charact eq} as \textit{characteristic roots}. For the special cases of dimension $n=2$, and dimension $n=2$ plus $B$ singular, by a direct computation equation \eqref{appendix linear system charact eq} becomes respectively
\begin{equation}\label{appendix linear system charact eq special cases}
\begin{split}
\lambda^2 - tr(A)\lambda + \det(A) + \text{e}^{-2\lambda\tau}\det(B) + \text{e}^{-\lambda\tau}\left[ C - \lambda tr(B) \right] &= 0,\\
\lambda^2 - tr(A)\lambda + \det(A) + \text{e}^{-\lambda\tau}\left[ C - \lambda tr(B) \right] &= 0,
\end{split}
\end{equation}
where $C = \det(a^1|b^2)+\det (b^1|a^2)$, and $a^i$ denotes the $i$-th  column  of $A$ (and the same for $b^i$). Thus, the  following lemma was proven:
\begin{lemma}
    For the following linear system of delay differential equations
\begin{equation}\label{appendix lemma linear system with one delay}
\dot{Z}(t) = AZ(t) + BZ(t-\tau),
\end{equation}
with  $A,B \in \mathbb{M}_{2\times 2}$ constant real matrices, $B$ singular, $Z=Z(t)\in \mathbb{R}^{2}$ and $\tau \geq 0$, the corresponding characteristic equation is
\begin{equation}\label{appendix lemma linear system charact eq}
\lambda^2 - tr(A)\lambda + \det(A) + \text{e}^{-\lambda\tau}\left[ C - \lambda tr(B) \right] = 0,
\end{equation}
where $C = \det(a^1|b^2)+\det (b^1|a^2)$, and $a^i$ denotes the $i$-th column of $A$ (same for $b^i$).
\end{lemma}
The stability of the trivial solution of system \eqref{appendix linear system with one delay} and the characteristic roots of equation \eqref{appendix linear system charact eq} are related by the following result (Theorem 4.3 in \cite{smith})
\begin{prop}\label{appendix prop stability characteristic roots}
    Suppose that $\text{Re} (\lambda) < \mu$ for every characteristic root $\lambda$. Then there exists $K>0$ such that
    \begin{equation}\label{appendix stability bound}
        |Z(t,\phi)| \leq K\text{e}^{\mu t} |\phi| \qquad t\geq 0,
    \end{equation}
    where $Z(t,\phi)$ is the solution of \eqref{appendix linear system with one delay} satisfying the initial condition $Z_0 = \phi \in C=C([-\tau,0], \R^n)$, where $C$ is endowed with the supreme norm.
\end{prop}
\begin{proof}
    See Theorem 4.3 in \cite{smith}.
\end{proof}
\begin{corollary}
    Looking at equation \eqref{appendix stability bound} of Proposition \ref{appendix prop stability characteristic roots}, we have the following direct consequences \begin{itemize}
        \item The trivial solution of system \eqref{appendix linear system with one delay} is asymptotically stable if $Re(\lambda) < 0$ for every characteristic root $\lambda$. (Note that the function defined by equation \eqref{appendix linear system charact eq} is analytic in $\C$. See sec. 4.3 in \cite{smith}.)
        \item The trivial solution of system \eqref{appendix linear system with one delay} is unstable if there is a characteristic root $\lambda$ satisfying $Re(\lambda) > 0$.
    \end{itemize}
\end{corollary}
\begin{remark}
    For the nonlinear system, the corresponding local stability results follow from the linearization principle (see sec. 4.6 in \cite{smith}) together with Proposition \ref{appendix prop stability characteristic roots}.
\end{remark}
\begin{remark}
    Equation \eqref{appendix linear system charact eq} can be written as $f(\lambda, \tau)=\lambda^n+g(\lambda,\tau)=0$ so in general, characteristic roots will change with the delay $\tau$. Thus, a stability switch may occur when a characteristic root crosses the imaginary axis as $\tau$ varies. We refer the reader to section \ref{subsec stability switches} for a more formal discussion of this result.
\end{remark}

\section*{Acknowledgments} The authors gratefully acknowledge the financial support provided by ICETEX through the program “Convocatoria Subvenciones para Proyectos de Internacionalización 2024.”



\bibliographystyle{abbrv}
\bibliography{bibliocrime}
\end{document}